\documentclass[11pt]{amsart}

\usepackage{amsmath,amssymb,amscd,amsthm,amsxtra,amsfonts}
\usepackage[dvips]{graphics,epsfig}
\usepackage[all]{xy}
\usepackage{url}

\allowdisplaybreaks[2]

\newtheorem{theorem}{Theorem} [section]

\newtheorem{lemma}[theorem]{Lemma}
\newtheorem{proposition}[theorem]{Proposition}
\newtheorem{remark}[theorem]{Remark} 

\newtheorem{definition}{Definition}
\newtheorem{corollary}[theorem]{Corollary}

\DeclareMathOperator*{\supp}{supp}

\newcommand{\noi}{\noindent}
\newcommand{\N}{\mathbb{N}}
\newcommand{\M}{\mathcal{M}}

\newcommand{\R}{\mathbb{R}}
\newcommand{\C}{\mathbb{C}}
\newcommand{\D}{\mathbb{D}}

\newcommand{\eps}{\varepsilon}

\newcommand{\ld}{\lambda}

\newcommand{\ft}{\widehat}

\newcommand{\cj}{\overline}
\newcommand{\dx}{\partial_x}
\newcommand{\dt}{\partial_t}

\newcommand{\LRA}{\Longrightarrow}

\newcommand{\jb}[1]
{\langle #1 \rangle}

\numberwithin{equation}{section}
\numberwithin{theorem}{section}

\begin{document}

\title
[Traveling Waves for the Cubic Szeg\"{o} Equation]
{\bf Traveling Waves for the Cubic Szeg\"{o} Equation on the Real Line}

\author{Oana Pocovnicu}

\address{Oana Pocovnicu\\
Laboratoire de Math\'ematiques d'Orsay\\
Universit\'e Paris-Sud (XI)\\
91405, Orsay Cedex, France}

\email{oana.pocovnicu@math.u-psud.fr}



\begin{abstract}
We consider the cubic Szeg\"{o} equation
\begin{equation*}
i\dt u=\Pi(|u|^{2}u)
\end{equation*}

\noi
in the Hardy space $L^2_+(\R)$ on the upper half-plane,
where $\Pi$ is the Szeg\"o projector.
It was first introduced by G\'erard and Grellier in \cite{PGSG}
as a toy model for totally non-dispersive evolution equations.
We show that the only traveling waves are of the form $\frac{C}{x-p}$,
$p \in \C$ with $\text{Im} p<0$.
Moreover, they are shown to be orbitally stable,
in contrast to the situation on the unit disk
where some traveling waves were shown to be unstable.
\end{abstract}

\date{\today}
\maketitle


\section{Introduction}

One of the most important properties
in the study of the nonlinear Schr\"odinger equations (NLS) is {\it dispersion}.
It is often exhibited in the form of the Strichartz estimates
of the corresponding linear flow.
In case of the cubic NLS:
\begin{equation}\label{eqn: Schrodinger}
i\partial_t u+\Delta u=|u|^2u, \quad (t,x)\in\R\times M,
\end{equation}

\noi
Burq, G\'erard, and Tzvetkov \cite{BGT} observed
that the dispersive properties are strongly influenced
by the geometry of the underlying manifold $M$.
Taking this idea further, G\'erard and Grellier \cite{PGSGX} remarked that
dispersion  disappears completely
when $M$ is a sub-Riemannian manifold
(for example, the Heisenberg group).
In this situation, many of the classical arguments
used in the study of NLS no longer hold.
As a consequence, even the problem of global well-posedness of \eqref{eqn: Schrodinger}
on a sub-Riemannian manifold still remains open.

In \cite{PGSG, PGSGX},
G\'erard and Grellier introduced a model of a non-dispersive Hamiltonian equation
called {\it the cubic Sz\"ego equation.} (See \eqref{eq:szego} below.)
The study of this equation is the first step
toward understanding existence and other properties of smooth solutions of NLS in the absence of dispersion.
Remarkably, the Szeg\"o equation turned out to be completely integrable
in the following sense.
It possesses a Lax pair structure
and an infinite sequence of conservation laws.
Moreover, the dynamics can be approximated
by a sequence of finite dimensional completely integrable Hamiltonian systems.
To illustrate the degeneracy of this completely integrable structure,
several instability phenomena were established in \cite{PGSG}.

The Sz\"ego equation was studied in \cite{PGSG, PGSGX} on the circle $\mathbb{S}^1$. More precisely,
solutions were considered to belong at all time to the Hardy space $L^2_+(\mathbb{S}^1)$ on the unit disk $\mathbb{D} = \{|z| < 1\}$.
This is the space of $L^2$-functions on $\mathbb{S}^1$
with $\hat{f}(k)=0$ for all $k<0$.
These functions can be extended as
holomorphic functions on the unit disk.
Several properties of the Hardy space on the unit disk
naturally transfer
to the Hardy space $L^2_+(\R)$ on the upper half-plane $\C_+ = \{ z; \text{Im} z > 0\}$,
defined by
\[L^2_+(\R)=\big\{f \text{ holomorphic on } \C_+;
\|g\|_{L^2_+(\R)}:=\sup_{y>0}\bigg(\int_{\R}|g(x+iy)|^{2}dx\bigg)^{1/2}<\infty  \big\}.\]

\noi
In view of the Paley-Wiener theorem, we identify
this space of holomorphic functions on $\C_+$
with the space of its boundary values:
\[L^2_+(\R)=\{f\in L^2(\R); \, \supp{\hat{f}}\subset [0,\infty)\}.\]

\noi
The transfer from $L^2_+(\mathbb{S}^1)$ to $L^2_+(\R)$
is made by the usual conformal transformation
$\omega:\D\to\C_+$ given by
\begin{equation*}
\omega(z)=i\frac{1+z}{1-z}.
\end{equation*}

\noi
However, the image of a solution of the Sz\"ego equation on $\mathbb{S}^1$
under the conformal transformation is no longer a solution of the Sz\"ego equation on $\R$.
Therefore, we directly study the Sz\"ego equation on $\R$ in the following.

Endowing $L^2(\R)$ with the usual scalar product $(u,v)=\int_{\R}u\bar{v}$, we define the Szeg\"{o} projector
$\Pi:L^{2}(\R)\to L^2_+(\R)$ to be the projector onto the non-negative frequencies,
\[\Pi (f)(x)=\frac{1}{2\pi}\int_{0}^{\infty}e^{ix\xi}\hat{f}(\xi)d\xi.\]

\noi
For $u \in L^2_+(\R)$,
we consider \textit{the Sz\"ego equation on the real line}:
\begin{equation}\label{eq:szego}
i\dt u=\Pi(|u|^{2}u), \quad x \in \R.
\end{equation}

\noi
This is a Hamiltonian evolution associated to the Hamiltonian
\begin{equation*}
E(u)=\int_{\R}|u|^4 dx
\end{equation*}

\noi
defined on $L^4_+(\R)$.
From this structure, we obtain the formal conservation law $E(u(t))=E(u(0))$.
The invariance under translations and under modulations provides two more conservation laws,
$Q(u(t))=Q(u(0))$ and $M(u(t))=M(u(0))$, where
\begin{equation*}
Q(u)=\int_{\R}|u|^2 dx\quad \text{and} \quad M(u)=\int_{\R}\bar{u}Du\,dx,\
\text{ with } D = -i \dx.
\end{equation*}

\noi
Now, we define the Sobolev spaces $H^s_+(\R)$ for $s\geq 0$:
\begin{align*}
H^s_+(\R)=&\big\{h\in L^2_+(\R);  \|h\|_{H^s_+}:=\bigg{(}\frac{1}{2\pi}\int_0^{\infty}(1+|\xi|^2)^{s}|\hat{h}(\xi)|^2d\xi\bigg{)}^{1/2}<\infty\big\}.
\end{align*}
Similarly, we define the homogeneous Sobolev norm for $h\in \dot{H}^s_+$ by
\begin{align*}
||h\|_{\dot{H}^s_+}:=\bigg{(}\frac{1}{2\pi}\int_0^{\infty}|\xi|^{2s}|\hat{h}(\xi)|^2\bigg{)}^{1/2}<\infty.
\end{align*}

Slight modifications of the proof of the corresponding result in \cite{PGSG} lead to the following well-posedness result:

\begin{theorem}
The cubic Szeg\"o equation \eqref{eq:szego} is globally well-posed in  $H_+^{s}(\R)$ for $s \geq \frac{1}{2}$,
i.e. given $u_0\in H_+^{1/2} $,
there exists a unique global-in-time solution $u\in C(\R; H_+^{1/2})$ of \eqref{eq:szego}
with initial condition $u_{0}$. Moreover, if $u_0\in H_+^{s}$ for some $s>\frac{1}{2}$,
then $u\in C(\R;H_+^{s})$.
\end{theorem}

In this paper, we concentrate on the study of traveling waves. The two main goals are the
classification of traveling waves and their stability.
As a result, we show that the situation on the real line is essentially different from that on the circle.

A solution for the cubic Szeg\"{o} equation on the real line \eqref{eq:szego}
is called a {\it traveling  wave} if there exist $c,\omega\in\R$ such that
\begin{equation} \label{eq:travel}
u(t, z)=e^{-i\omega t}u_0(z-ct), \quad z\in \C_+\cup\R, t \in\R
\end{equation}

\noi
for some $u_0 \in H^{1/2}_+(\R)$.
Note that a solution to \eqref{eq:szego} in $H^{1/2}_+(\R)$
has a natural extension onto $\C_+$,
and we have used this viewpoint in \eqref{eq:travel}.
Substituting \eqref{eq:travel} into \eqref{eq:szego},
 we obtain that $u_0$ satisfies the following equation on $\R$:
\begin{equation}\label{eqn:u}
cDu_0+\omega u_0=\Pi(|u_0|^{2}u_0).
\end{equation}

\noi
In the following,
we use the simpler notation $u$ instead of $u_0$,
when we study time-independent problems.
From \eqref{eqn:u}, we see that traveling waves with nonzero velocity, $c\neq 0$, have good regularity.
Indeed, we prove that $u\in H_+^{s}(\R)$ for all $s\geq 0$ in Lemma \ref{lemma: smothness}.
In particular, by Sobolev embedding theorem, we have $u\in L_+^p(\R)$ for $2\leq p\leq\infty$.
On the other hand, equation \eqref{eqn:u} yields in Lemma \ref{lemma: stationary waves} that there exist no nontrivial stationary waves, i.e. traveling waves of velocity $c=0$, in $L^2_+$.

Now, we present our main results:
\begin{theorem}\label{main th}
A function $u\in C(\R,H_+^{1/2}(\R))$ is a traveling wave if and only if there exist $C,p\in\C$
with $\textup{Im}\, p<0$ such that
\begin{equation} \label{eq:simple}
u(0,z)=\frac{C}{z-p}.
\end{equation}
\end{theorem}

\begin{theorem}\label{main cor}
Let $a>0$, $r>0$, and consider the cylinder
\[C(a,r)=\Big\{\frac{\alpha}{z-p}; |\alpha|=a, \textup{ Im}p=-r\Big\}.\]

\noi
Let $\{u_0^n\}\subset H^{1/2}_+$ with
\begin{equation*}
\inf_{\phi\in C(a,r)}\|u_0^n-\phi\|_{H^{1/2}_+}\to 0 \text{ as } n\to +\infty,
\end{equation*}

\noi
and let $u^n$ denote the solution to \eqref{eq:szego} with initial data $u_0^n$.
Then
\begin{equation*}
\sup_{t\in\R}\inf_{\phi\in C(a,r)}\|u^n(t,x)-\phi(x)\|_{H^{1/2}_+}\to 0.
\end{equation*}
\end{theorem}

Let us compare our results to those obtained in \cite{PGSG}.
In the case of the Szeg\"o equation on $\mathbb{S}^1$,
the nontrivial stationary waves ($c=0$) are finite Blaschke products of the form
\begin{equation*}
\alpha\prod_{j=1}^N\frac{z-p_j}{1-p_jz},
\end{equation*}

\noi
where $|\alpha|^2=\omega$, $N\in\N$, and $p_1,p_2,...,p_N\in\D$, and the traveling waves with nonzero velocity are
rational fractions of the form:
\begin{equation} \label{eq:simple2}
\frac{Cz^l}{z^N - p },
\end{equation}

\noi
where $N\in\N$, $l\in\{0,1,\dots,N-1\}$, $C,p\in\C$, and $|p|>1$.
Moreover, instability phenomena were displayed for some of the above traveling waves.
For the cubic Szeg\"o equation on $\R$,
Theorems \ref{main th} and \ref{main cor} state
that
there exist less traveling waves (corresponding to  $N = 1$ and $l = 0$ in \eqref{eq:simple2})
and that there is no instability phenomenon.

The proof of Theorem \ref{main th}
involves arguments from several areas of analysis: a Kronecker-type theorem, scattering theory,
existence of a Lax pair structure, a theorem by Lax on invariant subspaces of the Hardy space,
and canonical factorization of Beurling-Lax inner functions.
In the following, we  introduce the main notions and known results,
and
briefly describe the strategy of the proof.

As in \cite{PGSG}, an important property of the Sz\"ego equation on $\R$ is
the existence of a Lax pair structure.
Using the Szeg\"{o} projector, we first define two important classes of operators on $L^2_+$:
{\it the Hankel and Toeplitz operators}.
We use these operators to find a Lax pair.
See Proposition \ref{th:Lax pair}.

A Hankel operator $H_{u}:L^2_+\to L^2_+$ of symbol $u\in H^{1/2}_+$ is defined by
\[H_{u}(h)=\Pi(u\bar{h}).\]

\noi
Note that $H_{u}$ is $\C$-antilinear and satisfies
\begin{equation}\label{sym H_u}
(H_{u}(h_{1}),h_{2})=(H_{u}(h_{2}),h_{1}).
\end{equation}

\noi
In Lemma \ref{H_u is H-S} below we prove that $H_u$ is a Hilbert-Schmidt operator of Hilbert-Schmidt norm $\frac{1}{\sqrt{2\pi}}\|u\|_{\dot{H}^{1/2}}$.

A Toeplitz operator $T_{b}:L^2_+\to L^2_+$ of symbol $b\in L^{\infty}(\R)$ is defined by
\[T_{b}(h)=\Pi(bh).\]

\noi
$T_{b}$ is $\C$-linear.  Moreover, $T_b$ is self-adjoint if and only if $b$ is real-valued.

\begin{proposition}\label{th:Lax pair}
Let $u\in C(\R;H_+^{s})$ for some $s>\frac{1}{2}$.
The cubic Szeg\"{o} equation \eqref{eq:szego} is equivalent to
the following evolution equation:
\begin{equation}
\frac{d}{dt}H_{u}=[B_{u},H_{u}],
\end{equation}

\noi
where $B_{u}=\frac{i}{2}H_{u}^{2}-iT_{|u|^{2}}$.
In other words, the pair $(H_u,B_u)$ is a Lax pair for the cubic Szeg\"{o} equation on the real line.
\end{proposition}

The proof of Proposition \ref{th:Lax pair} follows the same
lines as that
of the corresponding result on $\mathbb{S}^1$ in \cite{PGSG}, and is based on the following identity:
\begin{equation}\label{Lax pair identity}
H_{\Pi(|u|^{2}u)}=T_{|u|^{2}}H_u+H_uT_{|u|^{2}}-H_u^3.
\end{equation}

\noi
Combining \eqref{eqn:u} and \eqref{Lax pair identity}, we deduce that if $u$ is a traveling wave with $c\neq 0$,
then the following identity holds:
\begin{equation}\label{identity for traveling waves}
A_uH_u+H_u A_u+\frac{\omega}{c} H_u+\frac{1}{c}H_u^{3}=0,
\end{equation}

\noi
where
\begin{equation}\label{eqn:A}
A_u=D-\frac{1}{c}T_{|u|^{2}}.
\end{equation}

In Section 2,
we prove a Kronecker-type theorem for the Hardy space $L^2_+(\R)$,
where we classify all the symbols $u$ such that the operator $H_u$ has finite rank. For a proof of the classical theorem for $L^2_+(\mathbb{S}^1)$, due to Kronecker, see \cite{PGSG}.

We prove Theorem \ref{main th} in Section 4.
We first prove that all the traveling waves are rational fractions.
On $\mathbb{S}^1$, this follows easily from the Kronecker theorem
and the fact that the operator $A_u$ has discrete spectrum.
On $\R$, however,
it turns out that $A_u$ has continuous spectrum.
Therefore, we use scattering theory to study the spectral properties of $A_u$
in detail in Section 3.
More precisely, we show that the generalized wave operators $\Omega^{\pm}(D,A_u)$,
rigorously defined by \eqref{eqn: wave operators} below,
exist and are complete. As a result, we obtain that
\[\mathcal{H}_{\text{ac}}(A_u)\subset \text{Ker} \, H_{u},\]

 \noi
 where $\mathcal{H}_{\text{ac}}(A_u)$ is the absolutely continuous subspace of $A_u$.
 The subspace $\text{Ker}\,  H_u$ plays an important role in our analysis.
 More precisely, it turns out to be invariant under multiplication
 by $e^{i\alpha x}$, for all $\alpha \geq 0$.
 Therefore, applying a theorem by Lax  (Proposition \ref{th: Lax 1959} below)
 on invariant subspaces, it results that
\[\text{Ker}\, H_u=\phi L^2_+,\]

\noi
 where $\phi$ is an inner function in the sense of Beurling-Lax,
 i.e. a bounded holomorphic function on $\C_+$ such that $|\phi(x)|=1$ for all $x\in\R$.
 Using the Lax pair structure and the identity \eqref{identity for traveling waves},
 we show that $\phi$ satisfies the following simple equation:
\begin{equation*}
cD\phi=|u|^2\phi.
\end{equation*}

\noi
However, as an inner function, $\phi$ satisfies a canonical factorization \eqref{eqn: factorization}.
From this,
it follows that $\phi$ belongs to a special class of inner functions, the finite Blaschke products, i.e.
\begin{equation*}
\phi(z)=\prod_{j=1}^N\frac{z-\ld_j}{z-\cj{\ld_j}},
\end{equation*}

\noi
where $N\in\N$ and $\text{Im} \ld_j>0$ for all $j=1,2,\dots,N$.
The Kronecker-type theorem then yields that the traveling wave $u$ is a rational fraction.

In the case of $\mathbb{S}^1$,
the natural shift, multiplication by $e^{ix}$,
was used
in concluding traveling waves are of the form \eqref{eq:simple2}.
In our case,
we use the ``infinitesimal" shift, multiplication by $x$,
to show that traveling waves are of the form \eqref{eq:simple}.

Finally, we prove Theorem \ref{main cor} in Section 5.
The orbital stability of traveling waves is a consequence of the fact
that traveling waves are ground states for the following inequality,
an analogue of Weinstein's sharp Gagliardo-Nirenberg inequality in \cite{Weinstein}.

\begin{proposition}\label{prop:GN}
For all $u\in H^{1/2}_{+}(\R)$ the following Gagliardo-Nirenberg inequality holds:
\begin{equation}\label{ineqn:Gagliardo-Nirenberg}
\|u\|_{L^{4}}\leq \frac{1}{\sqrt[4]{\pi}}\|u\|_{L^{2}}^{1/2}\|u\|_{\dot{H}_+^{1/2}}^{1/2},
\end{equation}
or, equivalently,
\begin{equation}\notag
E\leq \frac{1}{\pi}MQ.
\end{equation}

\noi
Moreover, equality holds if and only if $u=\frac{C}{x-p}$, where $C,p\in\C$ with $\textup{Im}p<0$.
\end{proposition}

\begin{remark}\rm
As a consequence of Proposition \ref{prop:GN}, one can verify that the functions $u=\frac{C}{x-p}$, with $\text{Im} p<0$, are indeed initial data for traveling waves.
More precisely,
since they are minimizers of the functional
\[v\in H^{1/2}_+\mapsto M(v)Q(v)-\pi E(v),\]
the differential of this functional at $u$ is zero. Thus,
\[\frac{1}{2}Q(u)Du+\frac{1}{2}M(u)u-\pi \Pi(|u|^2u)=0.\]
Consequently, $u$ is a solution of equation \eqref{eqn:u} with
\[c=\frac{Q(u)}{2\pi}=\frac{|C|^2}{-2\textup{Im} p},\,\,\,\,\,\,\,\,\, \omega=\frac{M(u)}{2\pi}=\frac{|C|^4}{4(-\textup{Im} p)^3}\]
and hence it is an initial datum for a traveling wave.
\end{remark}

In the case of $\mathbb{S}^1$, the Gagliardo-Nirenberg inequality suffices to conclude
the stability of the traveling waves with $N = 1$.
However, in the case of $\R$,
we need to use in addition a concentration-compactness argument.
This concentration-compactness argument,
which first appeared in the work of Cazenave and  Lions \cite{Cazenave},
was refined and turned into profile decomposition theorems by G\'erard \cite{PG}
and later by Hmidi and Keraani \cite{Keraani}.
We  use it in the form of Proposition \ref{prop: profile dec},
a profile decomposition theorem for bounded sequences in $H^{1/2}_+$.

We conclude this introduction by presenting two open problems.
Here, we use the term soliton instead of traveling wave, so that we put into light several connections with existing works.
The first problem is the soliton resolution,
which consists in writing any solution as a superposition of solitons and  radiation.
For the KdV equation,
this property was rigorously stated in \cite{KdV}
for initial data to which the Inverse Scattering Transform applies.
Therefore, for the Sz\"ego equation,
one needs to solve inverse spectral problems
for the Hankel operators and also find explicit action angle coordinates.

The second open problem is the interaction of solitons with external potentials.
Consider the Sz\"ego equation with a linear potential,
where initial data are taken to be of the form \eqref{eq:simple}.
As in the works of Holmer and Zworski \cite{Zworski}
and Perelman \cite{Perelman},
it would be interesting to investigate if solutions of
the perturbed Sz\"ego equation can be approximated by traveling wave solutions to
the original Szeg\"o equation \eqref{eq:szego}.

\section{A Kronecker-type theorem}
A theorem by Kronecker asserts in the setting of $\mathbb{S}^1$ that
the set of symbols $u$ such that $H_u$ is of rank $N$ is precisely a
$2N$-dimensional complex submanifold of $L^2_+(\mathbb{S}^1)$ containing
only rational fractions. In this section, we prove the analogue of this
theorem in the case of $L^2_+(\R)$. For a different proof of a similar
result on some Hankel operators
on $L^2_+(\R)$ defined in a slightly different way, we refer to Lemma 8.12, p.54 in \cite{Peller}.
\begin{definition}
Let $N\in\N^{\ast}$. We denote by $\M(N)$ the set of rational fractions of the form
$$\frac{A(z)}{B(z)},$$
where $A\in\C_{N-1}[z]$, $B\in\C_{N}[z]$, $0\leq \deg(A)\leq N-1$, $\deg(B)=N$, $B(0)=1$, $B(z)\neq 0$,
for all $z\in\C_{+}\cup\R$, and $A$ and $B$ have no common factors.
\end{definition}

\begin{theorem}\label{th: Kronecker}
The function $u$ belongs to $\M(N)$ if and only if the Hankel operator $H_u$ has complex rank N.

Moreover, if $u\in \M(N)$, $u(z)=\frac{A(z)}{B(z)}$, where $B(z)=\prod_{j=1}^{J}(z-p_j)^{m_j}$,
 with $\sum_{j=1}^Jm_j=N$ and $\textup{Im} p_j<0$ for all $j=1,2,...,J$,
 then the range of $H_u$ is given by
\begin{equation}\label{ran}
\textup{Ran}\,H_u=\textup{span}_{\C}\bigg\{\frac{1}{(z-p_j)^{m}}, 1\leq m\leq m_j\bigg\}_{j=1}^J
\end{equation}

\end{theorem}

\begin{proof}
The theorem will follow once we prove:

\begin{itemize}
\item[(i)] $u\in\M(N)\LRA \text{rk}(H_u)\leq N$

\item[(ii)]$\text{rk}(H_u)=N\LRA u\in\M(N).$
\end{itemize}

Let us first prove (i).
Let $u\in\M(N)$, i.e. $u$ is a linear combination of
$$\frac{1}{(z-p)^{m}},$$
where $\text{Im} p <0$, $1\leq m\leq m_{p}$, and $\sum m_{p}=N$.
Then, computing the integral
\begin{equation*}
\int_{\R}\frac{e^{-ix\xi}}{(x-p)^{m}}dx,
\end{equation*}

\noi
using the residue theorem, we obtain that $\hat{u}(\xi)=0$ for all $\xi\leq0$ and $\hat{u}(\xi)$
is a linear combination of $\xi^{m-1}e^{-ip\xi}$, with $1\leq m\leq m_{p}$, for $\xi> 0$.

Given $h\in L^2_+$,  we have $\widehat{H_{u}(h)}(\xi)=0$ for $\xi<0$.
Moreover, for $\xi>0$, we have
\begin{align}
\widehat{H_{u}(h)}(\xi)
& =\frac{1}{2\pi}\int_{-\infty}^{0}\hat{u}(\xi-\eta)\hat{\bar{h}}(\eta)d\eta \notag\\
& =\frac{1}{2\pi}\int_{0}^{\infty}\hat{u}(\xi+\eta)\overline{\hat{h}}(\eta)d\eta \label{H_u Fourier}\\
& =\sum_{\substack{1\leq m\leq m_{p} \notag
\\ \sum m_{p}=N}} c_{m,p}
\bigg(\sum_{k=0}^{m-1}C_{m-1}^{k}\xi^{m-1-k}\int_0^{\infty}\eta^{k}
\overline{\hat{h}}(\eta)e^{-ip\eta}d\eta \bigg)e^{-ip\xi}\notag\\
& =\sum_{\substack{1\leq m\leq m_{p}\notag\\ \sum m_{p}=N}}\tilde{d}_{m,p}(u,h)\xi^{m-1}e^{-ip\xi}
 =\sum_{\substack{1\leq m\leq m_{p} \notag\\ \sum m_{p}=N}}d_{m,p}(u,h)
 \bigg(\frac{1}{(x-p)^{m}}\bigg)^{\wedge}(\xi),\notag
\end{align}

\noi
where $c_{m,p}$, $\tilde{d}_{m,p}$, $d_{m,p}$ are constants depending on $p$ and $m$.

Hence,
\begin{equation}\label{eqn: Hu}
H_{u}(h)(x)=\sum_{\substack{1\leq m\leq m_{p}\\ \sum m_{p}=N}}\frac{d_{m,p}(u,h)}{(x-p)^{m}}
\end{equation}

\noi
and $\text{rk}(H_u)\leq N$.

Let us now prove (ii).
Assume that $\text{rank}(H_u)=N$, i.e.
the range of $H_u$, $\text{Ran}\, H_u$, is a 2$N$-dimensional real vector space.
As $H_u$ is $\C$-antilinear,
one can choose a basis of $\text{Ran}\, H_u$ of eigenvectors of $H_u$ in the following way:
\[\{v_1,iv_1,...,v_N,iv_N\, ; \,H_{u}(v_{j})=\ld_jv_j, \ld_j>0,j=1,2,\dots,N\}\]

\noi
Let $w_j=\sqrt{\ld}_jv_j.$ If $h\in L^2_+$, then by Parseval's identity we have
\begin{align*}
H_u(h)& =\sum_{j=1}^{N}(H_u(h),v_j)v_j+\sum_{j=1}^{N}(H_u(h),iv_j)iv_j=2\sum_{j=1}^{N}(H_u(h),v_j)v_j
=2\sum_{j=1}^{N}(H_u(v_j),h)v_j\\
&=2\sum_{j=1}^{N}(\ld_jv_j,h)v_j=2\sum_{j=1}^{N}(w_j,h)w_j
=\frac{1}{\pi}\sum_{j=1}^{N}\Big{(}\int_0^{\infty}\hat{w_j}(\eta)\overline{\hat{h}}(\eta)d\eta\Big{)} w_j.
\end{align*}

\noi
Consequently,
\begin{align*}
\widehat{H_{u}(h)}(\xi)
 =\frac{1}{2\pi}\pmb{1}_{\xi\geq 0}\int_{0}^{\infty}\hat{u}(\xi+\eta)\overline{\hat{h}}(\eta)d\eta
=\frac{1}{\pi}\pmb{1}_{\xi\geq 0}\sum_{j=1}^{N}\int_0^{\infty}\hat{w}_j(\eta)\hat{w}_j(\xi)\overline{\hat{h}}(\eta)d\eta.
\end{align*}

\noi
and hence,
\[\pmb{1}_{\xi\geq 0}\int_0^{\infty}\Big{(}\hat{u}(\xi+\eta)-2\sum_{j=1}^{N}\hat{w}_j(\eta)\hat{w}_j(\xi)\Big{)}\overline{\hat{h}}(\eta)d\eta=0,\]

\noi
for all $h\in L^2_+$. Therefore, for all $\xi,\eta\geq 0$, we have
\begin{align}\label{hat u(xi+eta)}
\hat{u}(\xi+\eta)=2\sum_{j=1}^{N}\hat{w}_j(\eta)\hat{w}_j(\xi).
\end{align}

\noi
Let $L>2N+1$ be an even integer and $\phi$ be the probability density function of the chi-square distribution defined by
\begin{align*}
\phi(\xi)=
\begin{cases}
&\frac{1}{2^{\frac{L}{2}}\Gamma (\frac{L}{2})}\xi ^{\frac{L}{2}-1}e^{-\frac{\xi}{2}}, \text{ if } \xi \geq 0\\
&0, \text{ if } \xi <0,
\end{cases}
\end{align*}

\noi
where $\Gamma$ is the Gamma function. Then, its Fourier transform is
\begin{align*}
\ft{\phi}(x)=(1+2ix)^{-\frac{L}{2}}.
\end{align*}

\noi
Notice that $\phi\in H^N(\R)$ since
\begin{align*}
\|\phi\|^2_{H^N}=\int_{\R} \frac{\jb{x}^{2N}}{|1+2i x|^{L}}dx
\end{align*}

\noi
which is convergent if and only if $2N-L<-1$.

Let $\jb{\theta,\psi}=\int_{\R}\theta(x)\psi(x)$ for all $\theta\in H^{-N}(\R)$ and $\psi\in H^N(\R)$. Consider the matrix $A_{\phi}$ defined by:
\[\left(
\begin{matrix}
\jb{\hat{w}_1,\phi} & \jb{\hat{w}_1',\phi} & \cdots & \jb{\hat{w}_1^{(N)},\phi} \\
\jb{\hat{w}_2,\phi} & \jb{\hat{w}_2',\phi} & \cdots & \jb{\hat{w}_2^{(N)},\phi}  \\
\vdots & \vdots  & \ddots & \vdots \\
\jb{\hat{w}_N,\phi} & \jb{\hat{w}_N',\phi} & \cdots & \jb{\hat{w}_N^{(N)},\phi}
\end{matrix}
\right)\]

\medskip
\noi

Since $\text{rk}(A_{\phi})\leq N$, it results that there exists $(c_0,c_1,\dots,c_N)\neq 0$ such that
\[\Big\langle \sum_{k=0}^{N}c_k\hat{w_j}^{(k)},\phi \Big\rangle=0,\]

\noi
for all $j=1,2,\dots,N$.
Then, since $\supp \phi\subset [0,\infty)$ and by \eqref{hat u(xi+eta)}, we have for all $\eta\geq 0$ that
\begin{align*}
\sum_{k=0}^{N}\Big\langle c_k\hat{u}^{(k)}(\xi),\phi(\xi-\eta)\Big\rangle_{\xi}&=\sum_{k=0}^{N}\Big\langle c_k\hat{u}^{(k)}(\xi+\eta),\phi(\xi)\Big\rangle_{\xi}=\sum_{k=0}^N(-1)^kc_k\int_0^{\infty}\hat{u}(\xi+\eta)\phi^{(k)}(\xi)d\xi\\
&=2\sum_{k=0}^N(-1)^kc_k\int_0^{\infty}\Big (\sum_{j=1}^N\hat{w}_j(\eta)\hat{w}_j(\xi)\Big)\phi^{(k)}(\xi)d\xi\\
&= 2\sum_{j=1}^{N}\hat{w}_j(\eta)\sum_{k=0}^{N}c_k\Big\langle\hat{w}_j^{(k)}(\xi),\phi(\xi)\Big\rangle=0.
\end{align*}

\noi
Denote $T=\sum_{k=0}^Nc_k\hat{u}^{(k)}$. Then $T\in H^{-N}$ and $\supp T\in [0,\infty)$. We have just proved that for all $\eta\geq 0$
\begin{align*}
0&=\jb{T,\phi(\cdot-\eta)}=\int_{\R}T(\xi)\phi(\xi-\eta)d\xi=\int_{\R}T(\xi)\Big (\int_{\R}\frac{e^{ix(\xi-\eta)}}{(1+2ix)^{L/2}}dx\Big )d\xi\\
&=\int_{\R}\Big (\int_{\R}T(\xi)e^{ix\xi}d\xi \Big )\frac{e^{-ix\eta}}{(1+2ix)^{L/2}}dx=\int_{\R}\mathcal{F}^{-1}T(x)\frac{e^{-ix\eta}}{(1+2ix)^{L/2}}dx.
\end{align*}

\noi
Denoting $R(x):=\frac{1}{(1+2ix)^{L/2}}\mathcal{F}^{-1}T(x)$, we have $\hat{R}\in H^{L/2-N}(\R)\subset H^{1/2}(\R)$ and
\begin{align*}
0=\int_{\R}R(x)e^{-ix\eta}dx=\hat{R}(\eta), \text { for all } \eta\geq 0.
\end{align*}

\noi
Thus $\supp \hat{R}\subset (-\infty, 0]$. By the definition of
$R$, $(1-2D_{\xi})^{L/2}\hat{R}(\xi)=T(\xi)$. Since the left hand-side
is supported on $(-\infty,0]$ and the right hand-side is supported
on $[0,\infty)$, we deduce that $\supp T\subset {0}$. In particular, $T_{|\xi>0}=0$.
This yields that $\hat{u}_{|\xi>0}$ is a weak solution on $(0,\infty)$
of the following linear ordinary differential equation:
\begin{equation}\label{equation hat{u}}
\sum_{k=0}^{N}c_kv^{(k)}(\xi)=0.
\end{equation}
\noi
Then, by \cite[Theorem 4.4.8, p.115]{Hormander}, we have that $\hat{u}_{|\xi>0}\in C^{N}((0,\infty))$,
$\hat{u}_{|\xi>0}$ is a classical solution of this equation
and therefore it is a linear combination of
\[\xi^{m-1}e^{q\xi}\]

\noi
where $q\in\C$ is a root of the polynomial $P(X)=\sum_{k=0}^{N}c_kX^{k}$ with multiplicity $m_q$,
$1\leq m\leq m_q$,
and $\sum_q m_{q}=N$.
Note that we must have $\text{Re} \, q<0$, because  $u\in L_+^{2}(\R)$.
Therefore we will denote $q=-ip$, with $\text{Im} \, p <0$ and obtain that
$\hat{u}(\xi)$ is a linear combination of $\xi^{m-1}e^{-ip\xi}$ for $\xi>0$. By the hypothesis $u\in L^2_+(\R)$, we obtain $\hat{u}(\xi)=0$ for $\xi\leq 0$. Hence for all $\xi\in\R$, $\hat{u}(\xi)$ is a linear combination of $\Big(\frac{1}{(x-p)^{m}}\Big)^\wedge(\xi)$, with $1\leq q\leq m_q$ and $\sum m_{q}=N$.
Thus $u\in\M(N')$ for some $N'\leq N$. If $N'<N$, then $(i)$ yields $\text{rk}(H_u)\leq N'$, which contradicts our assumption. In conclusion $u\in\M(N).$

Finally, when $u\in \M(N)$ we have $\text{rk}(H_u)=N$ and equation \eqref{eqn: Hu}, and thus \eqref{ran} follows.
\end{proof}
As a consequence of \eqref{ran} we make the following remark.
\begin{remark}\label{remark}
If $u\in \M(N)$, then $u\in\textup{Ran}\,H_u$.
\end{remark}
\section{Spectral properties of the operator $A_{u}$ for a traveling wave $u$}
Let us first recall the definition and the basic properties of the generalized wave operators, which are the main objects in scattering theory. We refer to chapter XI in \cite{Reed and Simon} for more details.

Let $A$ and $B$ be two self-adjoint operators on a Hilbert space $\mathcal{H}$. The basic principle of scattering theory is to compare the free dynamics corresponding to $e^{-iAt}$ and $e^{-iBt}$. The fact that $e^{-iBt}\phi$ "looks asymptotically free" as $t\to -\infty$, with respect to A, means that there exists $\phi_+\in \mathcal{H}$ such that
\begin{align*}
\lim_{t\to -\infty}\|e^{-iBt}\phi-e^{-itA}\phi_{+}\|=0
\intertext{or equivalently,}
\lim_{t\to -\infty}\|e^{iAt}e^{-itB}\phi-\phi_{+}\|=0.
\end{align*}
Hence, we reduced ourselves to the problem of the existence of a strong limit. Let $\mathcal{H}_{\text{ac}}(B)$ be the absolutely continuous subspace for $B$ and let $P_{\text{ac}}(B)$ be the orthogonal projection onto this subspace. In the definition of the generalized wave operators we have $\phi\in\mathcal{H}_{\text{ac}}(B)$.

We say that the generalized wave operators exist if the following strong limits exist:
\begin{equation}\label{eqn: wave operators}
\Omega^{\pm}(A,B)=\lim_{t\to\mp\infty}e^{itA}e^{-itB}P_{\text{ac}}(B).
\end{equation}

\noi
The wave operators $\Omega^{\pm}(A,B)$ are partial isometries with initial subspace $\mathcal{H}_{\text{ac}}(B)$ and with values
in $\text{Ran }\Omega^{\pm}(A,B)$. Moreover, $\text{Ran}\, \Omega^{\pm}(A,B)\subset \mathcal{H}_{\text{ac}}(A)$. If $\text{Ran}\, \Omega^{\pm}(A,B)=\mathcal{H}_{\text{ac}}(A)$, we say that the generalized wave operators are complete.

Lastly, we note that the following equality holds:
\begin{equation}\label{id: wave operators}
A\Omega^{\pm}(A,B)=\Omega^{\pm}(A,B)B.
\end{equation}

\begin{lemma}\label{lemma: smothness}
If $u\in H^{1/2}_+$ is a traveling wave, then $u\in H_+^{s}(\R)$ for all $s\geq 0$.
In particular, by Sobolev embedding theorem, we have $u\in L^p(\R)$ for $2\leq p\leq\infty$.
\end{lemma}

\begin{proof}
Because $u\in H^{1/2}(\R)$, the Sobolev embedding theorem yields $u\in L^p(\R)$, for all $2\leq p<\infty$. Therefore $|u|^2u\in L^{2}(\R)$ and thus $\Pi(|u|^2u)\in L^2_+$.
Using equation \eqref{eqn:u}
\begin{equation*}
cDu+\omega u=\Pi(|u|^{2}u),
\end{equation*}

\noi
we deduce that $Du\in L^2_+$. Consequently, $u\in H^1_+$ and by Sobolev embedding theorem we have $u\in L^{\infty}(\R)$.
Then $u^2D\bar{u}, |u|^2Du\in L^{2}(\R)$. Applying the operator $D$ to both sides of equation \eqref{eqn:u},
we obtain $D^2u\in L^2(\R)$ and hence $u\in H^2_+$. Iterating this argument infinitely many times, the conclusion follows.
\end{proof}

\begin{proposition}
Let $u$ be a traveling wave.
Then, $(A_u+i)^{-1}-(D+i)^{-1}$ is a trace class operator.
\end{proposition}

\begin{proof}
We prove first that for all $f\in L^2(\R)$,
the operator $(D+i)^{-1}f$, defined on $L^2(\R)$ by
\[\big{(}(D+i)^{-1}f\big{)}h(x)= (D+i)^{-1}(fh)(x)\]

\noi
is Hilbert-Schmidt. Denote by $\mathcal{F}$ the Fourier transform.
Using the isomorphism of $L^2(\R)$ induced by the Fourier transform,
 we have that $(D+i)^{-1}f$ is a Hilbert-Schmidt operator if and only if
$\mathcal{F}(D+i)^{-1}f$ is a Hilbert-Schmidt operator.
The latter is an integral operator of kernel
$K(\xi,\eta)=\frac{1}{2\pi}\cdot\frac{1}{\xi+i}\hat{f}(\xi-\eta)$.
Indeed,
\begin{equation*}
\mathcal{F}\big{(}(D+i)^{-1}fh\big{)}(\xi)=\frac{1}{2\pi}\cdot\frac{1}{\xi+i}\ft{fh}(\xi)
=\frac{1}{2\pi}\int_{\R}\frac{1}{\xi+i}\hat{f}(\xi-\eta)\hat{h}(\eta)d\eta=\int_{\R}K(\xi,\eta)\hat{h}(\eta)d\eta.
\end{equation*}

\noi
Therefore, it is Hilbert-Schmidt if and only if $K(\xi,\eta)\in L^2_{\xi,\eta}(\R\times\R)$. By the change of variables $\eta\mapsto \zeta=\xi-\eta$ we have
\begin{align*}
\|K(\xi,\eta)\|^2_{L^2_{\xi,\eta}}=\frac{1}{4\pi^2}\int_{\R}\frac{d\xi}{\xi^2+1}\int_{\R}|\hat{f}(\zeta)|^2d\zeta=C\|f\|_{L^2}^2<\infty.
\end{align*}

Hence $(D+i)^{-1}f$ is a Hilbert-Schmidt operator and so is $\bar{f}(D+i)^{-1}$, its adjoint. According to Lemma \ref{lemma: smothness}, $u\in L^{\infty}(\R)$ and thus $|u|^2\in L^{2}(\R)$. Taking $f=|u|^{2}$ and $f=u$, we conclude that the operators
 $(D+i)^{-1}|u|^{2}$, $(D+i)^{-1}u$, and $\bar{u}(D+i)^{-1}$ are all Hilbert-Schmidt.

We write
\begin{align*}
(A_u+i)^{-1}-(D+i)^{-1}
& =(D+i)^{-1}(D-A_u)(A_u+i)^{-1}\\
& =\frac{1}{c}(D+i)^{-1}T_{|u|^{2}}(A_u+i)^{-1}\\
& =\frac{1}{c}\Pi(D+i)^{-1}|u|^{2}(A_u+i)^{-1}=L(A_u+i)^{-1},
\end{align*}

\noi
where $L=\frac{1}{c}\Pi(D+i)^{-1}|u|^{2}$. Note that $L$ is a Hilbert-Schmidt operator since it is the composition of the bounded operator $\frac{1}{c}\Pi:L^2(\R)\to L^2_+$ with the Hilbert-Schmidt operator $(D+i)^{-1}|u|^{2}$.
Finally, we write, using the latter formula twice
\begin{align*}
(A_u+i)^{-1}-(D+i)^{-1}
& =L(L(A_u+i)^{-1}+(D+i)^{-1})\\
& =L\circ L\circ (A_u+i)^{-1}+\frac{1}{c}\Pi(D+i)^{-1}u\circ\bar{u}(D+i)^{-1}.
\end{align*}

\noi
We obtain that $(A_u+i)^{-1}-(D+i)^{-1}$ is a trace class operator since the composition of two Hilbert-Schmidt operators is a trace class operator.
\end{proof}

\begin{corollary}
If u is a traveling wave, then the wave operators $\Omega^{\pm}(D,A_u)$ exist and are complete.
\end{corollary}

\begin{proof}
This easily follows from  Kuroda-Birman theorem that we state below \cite[Theorem XI.9]{Reed and Simon}:

Let $A$ and $B$ be two self-adjoint operators on a Hilbert space such that
$(A+i)^{-1}-(B+i)^{-1}$ is a trace class operator. Then $\Omega^{\pm}(A,B)$ exist and are complete.
\end{proof}

\begin{corollary}\label{cor sigma ac}
If u is a traveling wave, then $\sigma_{\textup{ac}}(A_{u})=[0,+\infty)$.
\end{corollary}

\begin{proof}
Since $\Omega^{\pm}(D,A_u)$ are complete, it results that they are isometries
from $\mathcal{H}_{\text{ac}}(A_u)$ onto $\mathcal{H}_{\text{ac}}(D)=L^2_+$. By \eqref{id: wave operators}, we then have
\[{A_u}_{|_{\mathcal{H}_{\text{ac}}(A_u)}}
=[\Omega^{\pm}(D,A_u)_{|_{\mathcal{H}_{\text{ac}}(A_u)}}]^{-1}D\Omega^{\pm}(D,A_u)_{|_{\mathcal{H}_{\text{ac}}(A_u)}}.\]
Consequently, $\sigma_{\text{ac}}(A_{u})=\sigma_{\text{ac}}(D)=[0,+\infty)$.
\end{proof}

Our main goal in the following is to prove $\mathcal{H}_{\text{ac}}(A_{u})\subset\text{Ker}\, H_{u}$.
As we see below, it is enough to prove that $\big{[}\Omega^{+}(D,A_u)H^{2}_u\big{]}(\mathcal{H}_{\text{ac}}(A_{u}))=0$.

\begin{lemma}\label{H_u is H-S}
The operator $H_u$ is a Hilbert-Schmidt operator on $L^2_+(\R)$ of Hilbert-Schmidt norm $\frac{1}{\sqrt{2\pi}}\|u\|_{\dot{H}^{1/2}_+}$.
\end{lemma}

\begin{proof}
Let us denote by $\|T\|_{HS}$ the Hilbert-Schmidt norm of a Hilbert-Schmidt operator $T$.
By \eqref{H_u Fourier}, we have
\begin{equation*}
\widehat{H_{u}(h)}(\xi)
=\frac{1}{2\pi}\pmb{1}_{\xi\geq 0}\int_{0}^{\infty}\hat{u}(\xi+\eta)\overline{\hat{h}}(\eta)d\eta.
\end{equation*}

\noi
Then, we obtain
\begin{align*}
H_{u}(h)(x)
&=\frac{1}{4\pi^2}\int_{0}^{\infty}\int_{0}^{\infty}e^{ix\xi}\hat{u}(\xi+\eta)\overline{\hat{h}}(\eta)d\eta d\xi\\
&=\frac{1}{4\pi^2}\int_{\R}\bigg{(}\int_{0}^{\infty}
\int_{0}^{\infty}e^{ix\xi}e^{iy\eta}\hat{u}(\xi+\eta)d\eta d\xi\bigg{)}\bar{h}(y)dy.
\end{align*}

\noi
Using the fact that the Hilbert-Schmidt norm of an operator is equal to the norm of its integral kernel, Plancherel's formula, and Fubini's theorem, we have
\begin{align*}
\|H_{u}(h)\|^2_{HS}
&=\frac{1}{16\pi^4}\bigg\|\int_{0}^{\infty}\int_{0}^{\infty}e^{ix\xi}e^{iy\eta}\hat{u}(\xi+\eta)d\eta d\xi\bigg\|^2_{L^2_{x,y}}
=\frac{1}{4\pi^2}\|\pmb{1}_{\xi\geq 0}\pmb{1}_{\eta\geq 0}\hat{u}(\xi+\eta)\|^2_{L^2_{\eta,\xi}}\\
&=\frac{1}{4\pi^2}\int_{0}^{\infty}\int_{0}^{\infty}|\hat{u}(\xi+\eta)|^2d\eta d\xi=\frac{1}{4\pi^2}\int_{0}^{\infty}\int_{\xi}^{\infty}|\hat{u}(\zeta)|^2d\zeta d\xi\\
&=\frac{1}{4\pi^2}\int_{0}^{\infty}\bigg{(}\int_{0}^{\zeta}d\xi\bigg{)}|\hat{u}(\zeta)|^2d\zeta
=\frac{1}{4\pi^2}\int_{0}^{\infty}\zeta|\hat{u}(\zeta)|^2d\zeta =\frac{1}{2\pi}\|u\|_{\dot{H}^{1/2}}^2.
\end{align*}
\end{proof}

\begin{lemma}\label{lemma: Ker Hu^2}
$\textup{ Ker}\,H_u^2=\textup{Ker}\,H_u$. Moreover, if $\textup{ Ran} H_u$ is finite dimensional, then $\textup{ Ran}\,H_u^2=\textup{Ran}\,H_u$.
\end{lemma}

\begin{proof}
Let $f\in \text{Ker}\,H_u^2$. Then, by \eqref{sym H_u},
\begin{equation}\notag
(H_{u}(h_{1}),h_{2})=(H_{u}(h_{2}),h_{1}) \text{ for all } h_1,h_2\in L^2_+,
\end{equation}

\noi
we have
\[\|H_uf\|_{L^2}^2=(H_uf,H_uf)=(H_u^2f,f)=0\]

\noi
and thus $H_u f=0$.
Hence, $\text{Ker}\,H_u^2\subset\text{Ker}H_u\,$.
Therefore, we obtain $\text{Ker}\,H_u^2=\text{Ker}H_u$
since the inverse inclusion is obvious.

The identity \eqref{sym H_u}
yields also $\text{Ker} H_u=(\text{Ran} H_u)^{\perp}$. Moreover, it implies that $H_u^2$ is a self-adjoint operator and therefore, $\text{Ker} H^2_u=(\text{Ran} H^2_u)^{\perp}$. Hence, we obtain
\[(\text{Ran}H_u^2)^{\perp}=(\text{Ran}H_u)^{\perp}.\]

\noi
Taking the orthogonal complement of both sides, this yields
\[\overline{\text{Ran}H_u^2}=\overline{\text{Ran}H_u.}\]

\noi
If $\textup{Ran} H_u$ is finite dimensional, so is $\text{Ran}H_u^2$, since $\text{Ran}H_u^2\subset \text{Ran}H_u$.
Thus, $\text{Ran}H^2_u$ and $\text{Ran}H_u$ are closed.
Hence, we have $\text{Ran}H^2_u=\text{Ran}H_u$.
\end{proof}

\begin{lemma}\label{Ran Hu}
If $u$ is a traveling wave, then
\begin{equation}\label{eqn:comutativity A and H_u^2}
A_uH_u^2=H_u^2 A_u.
\end{equation}

\noi
Consequently, if $\textup{ Ran} H_u$ is finite dimensional, then $A_u(\textup{Ran}\, H_u)\subset \textup{Ran}\, H_{u}$.
\end{lemma}

\begin{proof}
The commutativity relation \eqref{eqn:comutativity A and H_u^2} is a consequence of identity \eqref{identity for traveling waves}. The second statement then follows by Lemma \ref{lemma: Ker Hu^2}, $\text{Ran}\,H_u^2=\text{Ran}\,H_u$.
\end{proof}

It is a classical fact that if $A$ and $B$ are two self-adjoint operators on a Hilbert space $\mathcal{H}$ such that $AB=BA$, then $B\big(\mathcal{H}_{\text{ac}}(A)\big)\subset\mathcal{H}_{\text{ac}}(A)$. For the sake of completeness, we prove it here in the case of the operators $A_u$ and $H_u^2$.

\begin{lemma}\label{lemma:commutativity H_u^2 and P_{ac}}
$H_u^2\mathcal{H}_{\textup{ac}}(A_u)\subset \mathcal{H}_{\textup{ac}}(A_u)$.
\end{lemma}

\begin{proof}
As we see below, the inclusion follows if we prove that $\mu_{H_u^{2}\phi}\ll\mu_{\phi}$ for all $\phi\in L^2_+$, where the measures above are the spectral measures with respect to the operator $A_u$, corresponding respectively to $H_u^2\phi$ and $\phi$.

Let $E\subset\R$ be a measurable set and $f=\pmb{1}_{E}$.
Then, by \eqref{eqn:comutativity A and H_u^2} and the Cauchy-Schwarz inequality we have
\begin{align*}
\mu_{H_{u}^2\phi}(E)
& =\int_{\R}fd\mu_{H_{u}^2\phi}
 =(H_{u}^2\phi,f(A_{u})H_{u}^2\phi)\\
& =(H_{u}^2\phi,H_{u}^2f(A_{u})\phi)=(H_{u}^4\phi,f(A_{u})\phi)\\
& \leq \sqrt{(f(A_{u})\phi,f(A_{u})\phi)}\|H_{u}^4\phi\|_{L^{2}}
=\sqrt{(\phi,f(A_{u})\phi)}\|H_{u}^4\phi\|_{L^{2}}\\
& =\sqrt{\mu_{\phi}(E)}\|H_{u}^4\phi\|_{L^{2}}.
\end{align*}

\noi
Therefore, $\mu_{H_u^{2}\phi}\ll\mu_{\phi}$.

Let us denote by $m$ the Lebesgue measure on $\R$.
 If $\phi\in \mathcal{H}_{\text{ac}}(A_u)$, then $\mu_{\phi}\ll m$ and thus $\mu_{H_u^{2}\phi}\ll m$. Hence, $H^2_u\mathcal{H}_{\text{ac}}(A_u)\subset\mathcal{H}_{\text{ac}}(A_u)$.
\end{proof}

\begin{proposition}\label{prop:H_ac subset Ker H_u}
If u is a traveling wave, then $\mathcal{H}_{\textup{ac}}(A_u)\subset \textup{Ker}\, H_u$.
\end{proposition}

\begin{proof}
It is enough to prove that $\big{[}\Omega^{+}(D,A_u)H_u^{2}\big{]}\big{(}\mathcal{H}_{\text{ac}}(A_u)\big{)}=0$.
If this holds, then we have $H_u^{2}\big{(}\mathcal{H}_{\text{ac}}(A_u)\big{)}=0$
since $H^2_u\mathcal{H}_{\text{ac}}(A_u)\subset\mathcal{H}_{\text{ac}}(A_u)$ and
$\Omega^{+}(D,A_u)$ is an isometry
on $\mathcal{H}_{\text{ac}}(A_u)$.
Therefore, $\mathcal{H}_{\text{ac}}(A_u)\subset\text{Ker}\, H_u^2=\text{Ker}\, H_u$.

Let us first note that
\begin{equation}\label{id}
H_ue^{itD}=e^{itD}H_{\tau_t(u)},
\end{equation}

\noi
where $\tau_{a}$ denotes the translation $\tau_{a}u(x)=u(x-a)$.
Indeed, for $f\in L^2_+$, passing into the Fourier space,
we have
\begin{align*}
\big{(}H_ue^{itD}f\big{)}^\wedge(\xi)&=\pmb{1}_{\xi\geq 0}\big{(}u\overline{e^{itD}f}\big{)}^\wedge(\xi)
=\frac{1}{2\pi}\pmb{1}_{\xi\geq 0}\int_{\R}\hat{u}(\xi-\eta)e^{it\eta}\hat{\bar{f}}(\eta)d\eta\\
&=\frac{1}{2\pi}\pmb{1}_{\xi\geq 0}e^{it\xi}\int e^{-it(\xi-\eta)}\hat{u}(\xi-\eta)\hat{\bar{f}}(\eta)d\eta
=\pmb{1}_{\xi\geq 0}e^{it\xi}\big{(}\tau_{t}(u)\bar{f}\big{)}^\wedge(\xi)\\
&=\pmb{1}_{\xi\geq 0}\big{(}e^{itD}(\tau_{t}(u)\bar{f})\big{)}^\wedge(\xi)=\big{(}e^{itD}H_{\tau_{t}(u)}f\big{)}^\wedge(\xi).
\end{align*}

By Lemma \ref{lemma:commutativity H_u^2 and P_{ac}}, \eqref{eqn:comutativity A and H_u^2}, and \eqref{id},
we have for all $f\in \mathcal{H}_{\text{ac}}(A_u)$
\begin{align*}
e^{itD}e^{-itA_{u}}P_{\text{ac}}H_u^2f
&=e^{itD}e^{-itA_{u}}H_u^2f
=e^{itD}H_u^2e^{-itA_{u}}f
=e^{itD}H_{u}H_ue^{-itD}e^{itD}e^{-itA_{u}}f\\
&=e^{itD}H_{u}e^{-itD}H_{\tau_{-t}(u)}e^{itD}e^{-itA_{u}}f
=H^2_{\tau_{-t}(u)}e^{itD}e^{-itA_{u}}P_{\text{ac}}(A_{u})f.
\end{align*}

\noi
We intend to prove that $H^2_{\tau_{-t}(u)}e^{itD}e^{-itA_{u}}P_{\text{ac}}(A_{u})f$ tends to 0 in the $L^2_+$-norm
as $t\to -\infty$. From this, we conclude that $\Omega^{+}(D,A_u)H_u^{2}f=0$.

Since, by Lemma \ref{H_u is H-S}, $H_{\tau_{-t}(u)}$ is a uniformly bounded operator,
 it is enough to prove that $H_{\tau_{-t}(u)}e^{itD}e^{-itA_{u}}P_{\text{ac}}(A_{u})f$ tends to 0.
\begin{align}
\|H_{\tau_{-t}(u)}e^{itD}e^{-itA_{u}}
& P_{\text{ac}}(A_{u})f\|_{L^2_+}\notag\\
&\leq \Big\|H_{\tau_{-t}(u)}\Big{(}e^{itD}e^{-itA_{u}}P_{\text{ac}}(A_{u})f-\Omega^{+}(D,A_u)f\Big{)}\Big\|_{L^2_+}\notag \\
& \hphantom{XXXXXX} +\|H_{\tau_{-t}(u)}\Omega^{+}(D,A_u)f\|_{L^2_+}\notag\\
&\leq \frac{1}{\sqrt{2\pi}}\|u\|_{\dot{H}^{1/2}}\|e^{itD}e^{-itA_{u}}P_{\text{ac}}(A_{u})f-\Omega^{+}(D,A_u)f\|_{L^2_+} \notag \\
& \hphantom{XXXXXX}+\int_{\R}|u(x+t)|^2|\Omega^{+}(D,A_u)f(x)|^2dx\label{terms}
\end{align}

\noi
The first term in \eqref{terms} converges to $0$ by the definition of the wave operator $\Omega^+(D,A_u)$.

Since $u$ is a traveling wave, \[u\in \bigcap_{s\geq 0}H^s(\R)\subset C^{\infty}_{\to 0}(\R),\]
 where $C^{\infty}_{\to 0}(\R)$ is the space of functions $f$ of class $C^{\infty}$ such that
$\lim_{x\to -\infty}D^{k}f(x)=\lim_{x\to \infty}D^{k}f(x)=0$
for all $k\in\N$. Therefore, for arbitrary fixed $x$, we have
\[\lim_{t\to -\infty}\tau_{-t}(u)(x)=\lim_{t\to -\infty}u(x+t)=0.\]
Note also that $|u(x+t)|^2|\Omega^{+}(D,A_u)f(x)|^2\leq \|u\|_{L^{\infty}}|\Omega^{+}(D,A_u)f(x)|^2$
 for all $x\in\R$. Then the second term in \eqref{terms} converges to $0$ by the dominated convergence theorem.
Hence $\big{[}\Omega^{+}(D,A_u)H_u^{2}\big{]}\big{(}\mathcal{H}_{\text{ac}}(A_u)\big{)}=0$.
\end{proof}

\section{Classification of traveling waves}

\begin{lemma}\label{lemma: stationary waves}
There are no nontrivial traveling waves of velocity $c=0$ in $L^2_+(\R)$.
\end{lemma}

\begin{proof}
Let $u$ be a nontrivial traveling wave of velocity $c=0$.
Then, equation \ref{eqn:u} gives $\Pi(|u|^2u)=\omega u$.
Taking the scalar product with $e^{i\xi x}u(x)$, where $\xi\geq 0$, we obtain
\begin{equation*}
\mathcal{F}(|u|^4-\omega |u|^2)(\xi)=0,
\end{equation*}

\noi
where $\mathcal{F}$ denotes the Fourier transform.
Since $|u|^4-\omega |u|^2$ is a real valued function,
we have that the last equality holds for all $\xi\in\R$.
Thus $|u|^4-\omega |u|^2=0$ on $\R$ and therefore $u(x)=0$ or $|u(x)|^2=\omega>0$,
for all $x\in\R$. Since the function $u$ is holomorphic on $\C_+$, its trace on $\R$ is either identically zero, or the set of zeros of $u$ on $\R$ has Lebesgue measure zero.
In conclusion, we have $|u|^2=\omega>0$ a.e. on $\R$ and thus $u$ is not a function in $L^2_+(\R)$.

\end{proof}

\begin{lemma}\label{lemma:uv holomorphic}
If $u\in H^{s}_+$ for $s>\frac{1}{2}$ and $v\in \textup{Ker } H_{u}$, then $\bar{u}v\in L^2_+$. Moreover, if $u\in L^{\infty}(\R)$, then $T_{|u|^{2}}v=|u|^{2}v$.
\end{lemma}

\begin{proof}
Indeed, $0=H_{u}(v)=\Pi(u\bar{v})$ and thus $\bar{u}v\in L^2_+$. Furthermore, since $u,\bar{u}v\in L^2_+$, we obtain $T_{|u|^{2}}v=\Pi(u\bar{u}v)=|u|^{2}v$.
\end{proof}

\begin{lemma}\label{lemma psi}
Let $u\in H^s_+$, $s>\frac{1}{2}$, be a solution of the cubic Szeg\"{o} equation \eqref{eq:szego}.
Consider the following Cauchy problem:
\begin{equation}\label{eqn psi}
\begin{cases}
 i\dt\psi=|u(t)|^{2}\psi\\
\psi \big|_{t = 0} = \psi_0,
\end{cases}
\end{equation}

\noi
If $\psi_{0}\in \textup{Ker}\, H_{u(0)}$,
then $\psi(t)\in \textup{Ker}\, H_{u(t)}$ for all $t\in\R$.

\begin{proof}
Let us first consider:
\begin{equation*}
\begin{cases}
 i\dt\psi_1=T_{|u(t)|^{2}}\psi_1\\
\psi_1 \big|_{t = 0} = \psi_0,
\end{cases}
\end{equation*}

\noi
Using the Lax pair structure,
 we have
\begin{align*}
\dt H_{u}(\psi_{1})&=[B_u,H_u]\psi_1+H_u\dt\psi_1=[\frac{i}{2}H_u^2-iT_{|u|^2},H_u]\psi_1+H_u(-iT_{|u|^2}\psi_1)\\
&=-iT_{|u|^2}H_u\psi_1-iH_uT_{|u|^2}\psi_1+iH_uT_{|u|^2}\psi_1=-iT_{|u|^2}H_u\psi_1.
\end{align*}

\noi
The solution of this linear Cauchy problem
\begin{equation*}
\begin{cases}
\dt H_{u}(\psi_{1})=-iT_{|u|^2}H_u\psi_1\\
H_u(\psi_1(0))=0
\end{cases}
\end{equation*}

\noi
is identically zero. i.e.,  $H_{u(t)}\psi_1(t)=0$ for all $t\in\R$.
Consequently, $\psi_1(t)\in \text{Ker}\,H_{u(t)}$ and by Lemma \ref{lemma:uv holomorphic} we obtain $T_{|u|^2}\psi_1=|u|^2\psi_1$. In conclusion, $\psi(t)=\psi_1(t)\in \text{Ker}\,H_{u(t)}$.
\end{proof}
\end{lemma}

The space $\text{Ker}\,H_{u}$ is invariant under multiplication by $e^{i\alpha x}$, for all $\alpha\geq 0$. Indeed, suppose $f\in\text{Ker}\,H_{u}$. Then $\ft{u\bar{f}}(\xi)=0$, for all $\xi\geq 0$ and
\[\big{(}H_u(e^{i\alpha x}f)\big{)}^\wedge(\xi)=\big{(}e^{-i\alpha x}u\bar{f}\big{)}^\wedge(\xi)=\ft{u\bar{f}}(\xi+\alpha)=0,\]
for all $\xi,\alpha\geq 0$. Hence, $e^{i\alpha x}f\in \text{Ker}\,H_u$ for all $\alpha\geq 0$.

One can then apply the following theorem to the subspaces $\text{Ker}\,H_{u_0}$.
\begin{proposition}[Lax \cite{Lax}]\label{th: Lax 1959}
Every non-empty closed subspace of $L^2_+$ which is invariant under multiplication by $e^{i\alpha x}$ for all $\alpha\geq 0$ is of the form $FL^2_+$, where $F$ is an analytic function in the upper-half plane, $|F(z)|\leq 1$ for all $z\in\C_+$, and $|F(x)|=1$ for all $x\in\R$.
Moreover, $F$ is uniquely determined up to multiplication by a complex constant of absolute value 1.
\end{proposition}
We deduce that $\text{Ker}\,H_{u_0}=\phi L^2_+$,
where $\phi$ is a holomorphic function in the upper half-plane $\C_+$,
satisfying $|\phi(x)|=1$ on $\R$ and $|\phi(z)|\leq 1$
for all $z\in\C_{+}$.

Functions satisfying the properties in Lax's theorem are called inner functions in the sense of Beurling-Lax.
A special class of inner functions is given by the Blaschke products. Given $\ld_{j}\in\C$ such that for all $j$
$$\text{Im}\,\ld_j>0$$
and
$$\sum_{j}\frac{\text{Im}\,\ld_j}{1+|\ld_j|^{2}}<\infty,$$
the corresponding Blaschke product is defined by
\begin{equation}\label {eqn: Blaschke}
B(z)=\prod_j\eps_j\frac{z-\ld_j}{z-\overline{\ld}_j},
\end{equation}

\noi
where $\eps_j=\tfrac{|\ld_j^2+1|}{\ld_j^2+1}$ (by definition $\eps_j=1$ if $\lambda_j=1$).

Inner functions have a canonical factorization, which is analogous to the canonical factorization of inner functions on the unit disk, see \cite[Theorem 17.15]{Rudin}, \cite[Theorem 6.4.4]{Nikolskii}. More precisely, every inner function $F$ can be written as the product
\begin{equation}\label{eqn: factorization}
F(z)=\ld B(z) e^{iaz} e^{i\int_{\R} \frac{1+tz}{t-z}d\nu(t)},
\end{equation}

\noi
where $z\in\C_+$, $\ld \in \C$ with $|\ld | = 1$, $a\geq 0$,
$B$ is a Blaschke product, and $\nu$ is a positive singular measure with respect to the Lebesgue measure.
In particular, the inner function $\phi$ has such a canonical factorization.

\begin{proposition}\label{lemma equation phi}
Let u be a traveling wave and denote by $\phi$
an inner function such that $\textup{Ker}\,H_{u_0}=\phi L^2_+$.
Then, $\phi$ satisfies the following equation on $\R$:
\begin{equation}\label{eqn:phi}
cD\phi=|u_{0}|^{2}\phi.
\end{equation}

\end{proposition}

\begin{proof}
Since $u(t,x)=e^{-i\omega t}u_{0}(x-ct)$,
we have $H_{u(t)}=e^{-i\omega t}\tau_{ct}H_{u_{0}}\tau_{-ct}$. Thus,
\begin{equation*}
\text{Ker}\,H_{u(t)}=\tau_{ct}\text{Ker}\,H_{u_{0}}=\tau_{ct}(\phi)L^2_+.
\end{equation*}

\noi
Let $f\in L^2_+$ and let $\psi_0=\phi f\in \text{Ker} H_{u_0}$ be the initial data of the Cauchy problem \eqref{eqn psi} in Lemma \ref{lemma psi}.
We then have $\phi e^{-i\int_{0}^{t}|u_s|^{2}ds}f\in \text{Ker} H_{u(t)}$.
Therefore,
\begin{equation}\label{inclusion}
\phi e^{-i\int_{0}^{t}|u_s|^{2}ds}L^2_+\subset\tau_{ct}(\phi) L^2_+.
\end{equation}

\noi
Conversely,
by solving backward the problem \eqref{eqn psi}
with the initial data in $\tau_{ct}(\phi) L^2_+$ at time $t$, up to the time $t=0$,
we obtain
\begin{equation*}
\tau_{ct}(\phi) L^2_+\subset\phi e^{-i\int_{0}^{t}|u_s|^{2}ds}L^2_+
\end{equation*}
\noi
and thus, the two sets are equal.

Let us first prove that $\phi_t:=\phi e^{-i\int_{0}^{t}|u_s|^{2}ds}$ is an inner function. Note that $\phi_t$ is well defined on $\R$ and its absolute value is 1 on $\R$. Consider the function defined by $h(x)=\frac{\phi_t(x)}{x+i}$, for all $x\in\R$. Since $h\in L^2_+$, we can write using the Poisson integral that
\begin{equation*}
h(z)=\frac{1}{\pi}\int_{-\infty}^{\infty}\text{Im}z\frac{h(x)}{|z-x|^2}dx,
\end{equation*}

\noi
for all $z\in\C_+$. Then,
\begin{equation*}
zh(z)=\frac{1}{\pi}\int_{-\infty}^{\infty}\text{Im}z\frac{xh(x)}{|z-x|^2}dx+\frac{1}{\pi}\int_{-\infty}^{\infty}\text{Im}z\frac{(z-x)h(x)}{|z-x|^2}dx.
\end{equation*}

\noi
Note that the last integral is equal to $\int_{-\infty}^{\infty}\text{Im}z\frac{h(x)}{\bar{z}-x}dx$. By the residue theorem and using the fact that the function $\frac{h}{\bar{z}-x}$ is holomorphic on $\C_+$, we have that this integral is zero and thus
\begin{equation*}
zh(z)=\frac{1}{\pi}\int_{-\infty}^{\infty}\text{Im}z\frac{xh(x)}{|z-x|^2}dx.
\end{equation*}

\noi
Therefore, we can use the Poisson integral to extend $\phi_t$ to $\C_+$ as a holomorphic function.
\begin{equation}\label{eqn:Poisson phi}
\phi_t(z)=(z+i)h(z)=\frac{1}{\pi}\int_{-\infty}^{\infty}\text{Im}z\frac{(x+i)h(x)}{|z-x|^2}dx=\frac{1}{\pi}\int_{-\infty}^{\infty}\text{Im}z\frac{\phi_t(x)}{|z-x|^2}dx.
\end{equation}

\noi
Moreover,
\begin{equation*}
|\phi_t(z)|\leq\frac{1}{\pi}\int_{-\infty}^{\infty}\text{Im}z\frac{1}{|z-x|^2}dx=1,
\end{equation*}

\noi
for all $z\in\C_+$. Hence $\phi_t$ is an inner function.

Since $\tau_{ct}(\phi) $ and $\phi e^{-i\int_{0}^{t}|u_s|^{2}ds}$ are inner functions and
\begin{equation*}
\phi e^{-i\int_{0}^{t}|u_s|^{2}ds}L^2_+=\tau_{ct}(\phi) L^2_+,
\end{equation*}

\noi
Proposition \ref{th: Lax 1959} yields the existence of a real valued function $\gamma$ such that $\gamma(0)=0$ and
$$\phi e^{-i\int_{0}^{t}|u_s|^{2}ds}=\tau_{ct}(\phi) e^{i\gamma(t)}.$$
\noi
Taking the derivative with respect to $t$ we obtain that $\phi$ satisfies the following equation:
\begin{equation*}
cD\phi(x)=|u(t,x+ct)|^{2}\phi(x)+\dot{\gamma}(t)\phi(x).
\end{equation*}

\noi
for all $t\in\R$.
Since $u$ is a traveling wave, we have $|u(t,x+ct)|=|e^{-i\omega t}u_0(x)|=|u_0(x)|$.
Then we deduce that $\dot{\gamma}(t)=k$ and hence $\gamma (t)=kt$, for some $k\in\R$. Therefore,
\begin{equation}\label{eqn:phi 2}
cD\phi=(|u_{0}|^{2}+k)\phi.
\end{equation}

We prove in the following that $k=0$.
First, note that $\frac{k}{c}\geq 0$.
The function $\phi u_0\in \text{Ker} H_{u_0}$ and by Lemma \ref{lemma:uv holomorphic}, we have $|u_0|^2\phi=\overline{u}_0(u_0\phi)\in L^2_+$.
If $\frac{k}{c}$ is negative,
denoting $\chi:=\frac{1}{c}|u_0|^{2}\phi\in L^{2}_+$ and passing into the Fourier space,
we have:
\begin{equation*}
\hat{\phi}(\xi)=\frac{1}{\xi-\frac{k}{c}}\hat{\chi}(\xi) \, \pmb{1}_{[0, \infty)}(\xi).
\end{equation*}

\noi
This implies that $\phi\in L^{2}_+$, contradicting $|\phi(x)|=1$ for all $x\in\R$.

Let us now prove that $\frac{k}{c}=0$.
Let $h\in L^2_+$ regular.
Then $\phi h\in \text{Ker}\, H_{u_0}$ and by equation \eqref{eqn:phi 2} we have
\begin{equation*}
A_{u_0}(\phi h)=(D-\tfrac{1}{c}|u_0|^{2})(\phi h)=\phi(D-\tfrac{1}{c}|u_0|^{2})(h)+hD\phi=\phi(D+\tfrac{k}{c})h.
\end{equation*}

\noi
Denoting by $\mu_{\phi h}(A_{u_0})$ the spectral measure corresponding to $\phi h$,
we have
\begin{align*}
\int f d\mu_{\phi h}
& =(\phi h, f(A_{u_0}) \phi h)=(\phi h, \phi f(D+\tfrac{k}{c}) h)=(h,f(D+\tfrac{k}{c}) h)\\
& =\frac{1}{2\pi}\int_{0}^{\infty} f(\xi+\tfrac{k}{c})|\hat{h}(\xi)|^{2} d\xi
=\frac{1}{2\pi}\int_{\frac{k}{c}}^{\infty} f(\eta)|\hat{h}(\eta-\tfrac{k}{c})|^{2} d\eta.
\end{align*}

\noi
Consequently, $\supp \mu_{\phi h}(A_{u_0})\subset [\frac{k}{c},+\infty)$. By Proposition \ref{prop:H_ac subset Ker H_u}, we have $\mathcal{H}_{\text{ac}}(A_{u_0})\subset\text{Ker} \, H_{u_0}$, and therefore
\begin{equation*}
\sigma_{\text{ac}}(A_{u_0})
=\overline{\bigcup_{\psi\in \mathcal{H}_{\text{ac}}(A_{u_0})}\supp\mu_{\psi}}
\subset\overline{\bigcup_{\phi h\in \text{Ker} H_{u_0}}\supp\mu_{\phi h}}\subset \big[\tfrac{k}{c},\infty\big).
\end{equation*}

\noi
Since, by Corollary \ref{cor sigma ac}, $\sigma_{\textup{ac}}(A_{u_{0}})=[0,\infty)$, this yields $k=0$.
\end{proof}

\begin{proposition}
All traveling waves are rational fractions.
\end{proposition}

\begin{proof}
We first prove that $\phi$ is a Blaschke product.

Since $\phi$ is an inner function in the sense of Beurling-Lax, it has the following canonical decomposition:
\begin{equation}\label{eqn: phi intermediate}
\phi(z)=\ld B(z) e^{iaz} e^{i\int_{\R} \frac{1+tz}{t-z}d\nu(t)},
\end{equation}

\noi
where $z\in\C_+$, $\ld$ is a complex number of absolute value 1, $a\geq 0$,
$B$ is a Blaschke product having exactly the same zeroes as $\phi$,
and $\nu$ is a positive singular measure with respect to the Lebesgue measure.

Because $\phi$ satisfies the equation \eqref{eqn:phi} and $u_0\in L^{\infty}(\R)$,
we obtain that $\phi$ has bounded derivative on $\R$ and hence it is uniformly continuous on $\R$. Then,
since $\phi$ satisfies the Poisson formula \eqref{eqn:Poisson phi},
it follows that
\begin{equation*}
\phi(x+i\eps)\to\phi(x), \text{ as }\eps\to 0,
\end{equation*}

\noi
uniformly for $x\in\R$.

$\phi$ being uniformly continuous on $\R$ and $|\phi(x)|=1$, $\forall x\in\R$, we deduce that the zeroes of $\phi$ and hence,
those of the Blaschcke product $B$ as well, lie outside a strip $\{z\in\C; 0\leq \textup{Im} z\leq \eps_0\}$, for some $\eps_0>0$.
Therefore, we have
\begin{equation*}
\frac{\phi(x+i\eps)}{B(x+i\eps)}\to \frac{\phi(x)}{B(x)}, \text{ as } \eps\to 0
\end{equation*}

\noi
uniformly for $x$ in compact subsets of $\R$. Taking the logarithm of the absolute value
and noticing that $\big{|}\frac{\phi(x)}{B(x)}\big{|}=1$, we obtain
\begin{equation*}
\int_{\R}\frac{\eps}{(x-t)^2+\eps^2}d\nu(t)\to 0,
\end{equation*}

\noi
uniformly for $x$ in compact subsets in $\R$.
In particular, for all $\delta>0$ there exists $0<\eps_1\leq \eps_0$ such that for all $0<\eps\leq \eps_1$ and for all $x\in [0,1]$, we have
\begin{equation*}
\frac{1}{2\eps}\nu([x-\eps,x+\eps])\leq\int_{x-\eps}^{x+\eps}\frac{\eps}{(x-t)^2+\eps^2}d\nu(t)\leq\int_{\R}\frac{\eps}{(x-t)^2+\eps^2}d\nu(t)\leq \delta.
\end{equation*}

\noi
Taking $\eps=\frac{1}{2N}\leq \eps_1$ with $N\in\N^{\ast}$, we obtain
$$\nu([0,1])=\nu\big(\bigcup_{k=0}^{N-1}[\frac{k}{N},\frac{k+1}{N}]\big)\leq N\delta\frac{1}{N}=\delta.$$

\noi
In conclusion $\nu([0,1])=0$, and one can prove similarly that the measure $\nu$ of any compact interval in $\R$ is zero. Hence $\nu\equiv 0$.

Consequently, $\phi(x)=\ld B(x) e^{iax}$ for all $x\in\R$.
On the other hand, because $\phi$ satisfies the equation \eqref{eqn:phi},
we have  $\phi(x)=\phi(0)e^{\frac{i}{c}\int_0^x|u_0|^{2}}$ and, in particular, $\lim_{x\to\infty}\phi(x)=\phi(0)e^{\frac{i}{c}\int_0^{\infty}|u_0|^{2}}$.
Since $\lim_{x\to\infty}B(x)=1$, we conclude that $a=0$. Substituting $\phi=\ld B$ in the equation \eqref{eqn:phi},
 we obtain
$$\frac{c}{i}\frac{B'}{B}=|u_0|^{2}.$$
Then $\frac{1}{i}\int_{-\infty}^{\infty}\frac{B'(x)}{B(x)}\,dx<\infty$.
Computing this integral, we obtain that
\[\frac{1}{i}\int_{-\infty}^{\infty}\frac{B'(x)}{B(x)}\,dx=2\sum_{j}\int_{-\infty}^{\infty}\frac{\text{Im} \ld_j}{|x-\ld_j|^2}\,dx=2\sum_{j}\pi\]

\noi
and thus it is finite if and only if $B$ is a finite Blaschke product,
$B(x)=\prod_{j=1}^{N} \eps_{j}\frac{x-\ld_{j}}{x-\overline{\ld}_{j}}$.

Let us prove that the traveling wave $u$ is a rational fraction.
\[\text{Ker}\,H_{u}=\phi L^2_+=B L^2_+.\]

\noi
Notice that
$B L^2_+=\Big{(}\text{span}_{\C}\Big\{\frac{1}{x-\overline{\ld}_{j}}\Big\}_{j=1}^N\Big{)}^{\perp}.$
Indeed, $f\in \Big{(}\text{span}_{\C}\Big\{\frac{1}{x-\overline{\ld}_{j}}\Big\}_{j=1}^N\Big{)}^{\perp}$
if and only if
\[f(\ld_{j})=\frac{1}{2\pi}\int_{\R}e^{i\xi\ld_j}\ft{f}(\xi)d\xi=\frac{1}{2\pi}\Big(\hat{f},e^{-i\overline{\ld}_{j}\xi}\Big)
=\Big(f,\frac{1}{x-\overline{\ld}_{j}}\Big)=0,\]

\noi
if and only if there exists $h\in L^2_+$ such that $f=Bh$.
Hence \[\text{Ker}\,H_{u}=\bigg{(}\text{span}_{\C}\bigg\{\frac{1}{x-\overline{\ld}_{j}}\bigg\}_{j=1}^N\bigg{)}^{\perp}\]
This yields $\overline{\text{Ran}\, H}_{u}=\text{span}_{\C}\Big\{\frac{1}{x-\overline{\ld}_{j}}\Big\}_{j=1}^N$. By Remark \ref{remark} it follows that $u$ is a rational fraction. More precisely, $u\in\text{Ran}\,H_u =\text{span}_{\C}\Big\{\frac{1}{x-\overline{\ld}_{j}}\Big\}_{j=1}^N$.
\end{proof}

\begin{proposition}\label{prop: Hu^2=ld u}
If $u$ is a traveling wave, then there exists $\ld>0$ such that $H_{u}^{2}u=\ld u$.
\end{proposition}

\begin{proof}
According to Remark \ref{remark}, since $u$ is a rational fraction, we have $u\in\text{Ran }H_u$.

Secondly, $u$ satisfies the equation of the traveling waves \eqref{eqn:u}, which is equivalent to $A_u(u)=-\frac{\omega}{c} u$. Therefore, $u$ is an eigenfunction of the operator $A_u$ for the eigenvalue $-\frac{\omega}{c}$.
Applying the identity \eqref{identity for traveling waves},
\begin{equation*}
A_uH_{u}+H_{u}A_u+\frac{\omega}{c} H_{u}+\frac{1}{c}H_{u}^{3}=0,
\end{equation*}

\noi
to $u$ and then to $H_u u$, one deduces that $A_uH_{u}^{2}u=-\frac{\omega}{c} H_{u}^{2}u$.
Therefore, the conclusion of the proposition follows
once we prove all the eigenfunctions of the operator $A_u$ belonging to $\text{Ran}\,H_u$, corresponding to the same eigenvalue, are linearly dependent.

Let $a$ be en eigenvalue of the operator $A_u$ and let $\psi_{1},\psi_{2}\in \text{Ker}\,(A_u-a)\cap \text{Ran} \, H_{u}$.
Since $u$ is a rational fraction, by the Kronecker type theorem \ref{th: Kronecker}, $\psi_1$ and $\psi_2$ are also nonconstant rational fractions. Then, one can find $\alpha,\beta\in\C$, $(\alpha,\beta)\neq (0,0)$, such that $\psi:=\alpha\psi_{1}+\beta\psi_{2}=O(\frac{1}{x^{2}})$ as $x\to\infty$.
Moreover, we have $\psi\in L^1(\R)$, $x\psi\in L^{2}(\R)$, and thus we can compute $A_u(x\psi)$.

Passing into the Fourier space we have,
\begin{equation*}
\ft{\Pi(xf)}(\xi)=i(\partial_{\xi}\hat{f})\pmb{1}_{\xi\geq 0}=i\partial_{\xi}(\hat{f}\pmb{1}_{\xi\geq 0})-i\hat{f}(\xi)\delta_{\xi=0}=\ft{x\Pi f}(\xi)-i\hat{f}(0)\delta_{\xi=0},
\end{equation*}

\noi
for all $f\in L^1(\R)$. Thus, we obtain $\Pi(xf)=x\Pi(f)+\frac{1}{2\pi i}\hat{f}(0)$ for all $f\in L^{1}(\R)$.
We then have
\begin{equation*}
A_u(x\psi)=xA_u(\psi)+\frac{1}{i}\psi-\frac{1}{2c\pi i}\int_{\R}|u|^{2}\psi dx
\end{equation*}
and therefore, since $A_u\psi=a\psi$,
\begin{equation}\label{eqn: x psi}
A_u(x\psi)=ax\psi+\frac{1}{i}\psi-\frac{1}{2c\pi i}\int_{\R}|u|^{2}\psi dx.
\end{equation}

\noi
Since $x\psi\in \text{Ran}\, H_{u}$ and $A_u(\text{Ran}\,H_u)\subset\text{Ran}\,H_u$ by Lemma \ref{Ran Hu},
we have $A_u(x\psi)\in \text{Ran}\, H_{u}\subset L^{2}(\R)$.
The constant in equation \eqref{eqn: x psi} is zero because all the other terms are in $L^2(\R)$.
Then we have
\begin{equation}\label{eqn: x psi 2}
(A_u-a)(x\psi)=\frac{1}{i}\psi.
\end{equation}

\noi
Applying  the self-adjoint operator $A_u-a$ to both sides of the equation \eqref{eqn: x psi 2}, we obtain $(A_u-a)^2(x\psi)=0$ and
\begin{equation*}
\|(A_u-a)(x\psi)\|_{L^2}^2=((A_u-a)(x\psi),(A_u-a)(x\psi))=((A_u-a)^2(x\psi),x\psi)=0.
\end{equation*}

\noi
Thus, $(A_u-a)(x\psi)=0$. In conclusion, by equation \eqref{eqn: x psi 2}, $\psi=0$ and therefore all the eigenfunctions belonging to $\text{Ran}\,H_u$, corresponding to the same eigenvalue $a$, are linearly dependent.
\end{proof}

\begin{proof}[Proof of Theorem \ref{main th}]
Since $u\in\text{Ran} \, H_{u}$,
there exists a unique function $g\in \text{Ran}\, H_{u}$ such that $u=H_{u}(g)$.
By Lemma \ref{prop: Hu^2=ld u}, it results that $H_{u}(u)=\lambda g$.
Applying the identity \eqref{identity for traveling waves},
\begin{equation*}
A_uH_{u}+H_{u}A_u+\frac{\omega}{c} H_{u}+\frac{1}{c}H_{u}^{3}=0,
\end{equation*}

\noi
to $g$ and using $A_uu=-\frac{\omega}{c}u$, one obtains $H_{u}(A_ug+\frac{\ld}{c} g)=0$.
Since $A_u(\text{Ran}\, H_{u})\subset \text{Ran}\, H_{u}$, we have $A_ug+\frac{\ld}{c} g\in \text{Ran}\, H_{u}\cap\text{Ker}\,H_u$. Therefore, $A_ug+\frac{\ld}{c} g=0$, which is equivalent to
\begin{equation*}
cDg-T_{|u|^{2}}g+\ld g=0.
\end{equation*}

\noi
In the following we intend to find a simpler version of the above equation, in order to determine the function $g$ explicitely.
Note that $\bar{u}(1-g)\in L^2_+$, since it is orthogonal to each complex conjugate of a holomorphic function $f\in L^2_+$:
\begin{equation*}
(\bar{u}(1-g),\bar{f})=(f(1-g),u)=(f,u)-(f,H_{u}(g))=0.
\end{equation*}

\noi
Thus, $T_{|u|^{2}}(g)=\Pi(|u|^{2})-\Pi(|u|^{2}(1-g))=H_{u}(u)-|u|^{2}(1-g)=\ld g-|u|^{2}(1-g).$

Passing into the Fourier space and using the fact that $|u|^{2}$ is a real valued function, one can write
$$|u|^{2}=\int_0^{\infty}e^{ix\xi}\ft{|u|^2}(\xi)d\xi+\int_0^{\infty}e^{-ix\xi}\overline{\ft{|u|^2}}(\xi)d\xi=\Pi(|u|^{2})+\overline{\Pi(|u|^{2})}.$$
\noi
Therefore $|u|^{2}=H_{u}(u)+\cj{H_{u}(u)}=\ld(g+\cj{g})$.
Consequently, $T_{|u|^{2}}(g)=\ld (-\bar{g}+g^{2}+|g|^{2})$ and $g$ solves the equation
\begin{equation} \label{eqn:g}
cDg-\ld g^{2}+\ld(g+\bar{g}-|g|^{2})=0.
\end{equation}

We prove that $g+\bar{g}-|g|^{2}=0$.
First, note that $\bar{u}(1-g)\in L^2_+$ , also yields $(1-g)f\in \text{Ker} \, H_{u}$, for all $f\in L^2_+$.
Secondly, let us prove that $g+\bar{g}-|g|^{2}$ is orthogonal to the complex conjugate of all $f\in L^2_+$:
\[(g+\bar{g}-|g|^{2},\bar{f})=(g,\bar{f})-(f(1-g),g)=-(f(1-g),\tfrac{1}{\ld}H_{u}(u))=-\tfrac{1}{\ld}(u,H_{u}(f(1-g)))=0.\]
In addition, since $g+\bar{g}-|g|^{2}$ is a real valued function, we have
\[(g+\bar{g}-|g|^{2},f)=(g+\bar{g}-|g|^{2},\bar{f})=0\]
for all $f\in L^2_+$. Therefore, $g+\bar{g}-|g|^{2}$ is orthogonal to all the functions in $L^2(\R)$ and thus $g+\bar{g}-|g|^{2}=0$.
This is equivalent to $|1-g|=1$ on $\R$. Moreover, equation \eqref{eqn:g} gives the precise formula for $g$,
\begin{equation*}
g(z)=\frac{r}{z-p},
\end{equation*}
where $r,p\in\C$ and $\text{Im}(p)<0.$
Thus $1-g(x)=\frac{x-\bar{p}}{x-p}$ for all $x\in\R$ and
\begin{equation*}
\text{Ker} \, H_{\frac{1}{z-p}}=\frac{z-\bar{p}}{z-p}L^2_+=(1-g)L^2_+\subset \text{Ker}\, H_{u}.
\end{equation*}

\noi
Consequently, $u\in\text{Ran}\, H_{u}\subset \text{Ran} \, H_{\frac{1}{z-p}}=\frac{\C}{z-p}$.
\end{proof}

\section{Orbital stability of traveling waves}

In order to prove the orbital stability of traveling waves, we first use the fact that they are minimizers of the Gagliardo-Nirenberg inequality. We begin this section by proving this inequality, more precisely proposition \ref{prop:GN}.

\begin{proof}[Proof of Proposition \ref{prop:GN},  Gagliardo-Nirenberg inequality]
The proof is similar to the proof of  Gagliardo-Nirenberg inequality for the circle, in \cite{PGSGX}. The idea is to write all the norms in the Fourier space, using Plancherel's identity.
\begin{equation*}
E=\|u\|_{L^{4}}^{4}=\|u^{2}\|_{L^{2}}^2=\frac{1}{2\pi}\big\|\ft{u^{2}}\big\|_{L^{2}}^2=\frac{1}{2\pi}\int _{\R}|\ft{u^2}(\xi)|^2d\xi.
\end{equation*}

\noi
Using the fact that $u\in L^2_+$ and Cauchy-Schwarz inequality, we have:
\begin{align*}|\ft{u^2}(\xi)|^2&=\frac{1}{4\pi^2}\big{|}\int_0^{\xi}\ft{u}(\eta)\ft{u}(\xi-\eta)d\eta\big{|}^2\leq \frac{1}{4\pi^2}\xi\int_0^{\xi}|\ft{u}(\eta)|^2|\ft{u}(\xi-\eta)|^2d\eta\\
&\leq\frac{1}{4\pi^2}\Big{(}\int_0^{\xi}\eta|\ft{u}(\eta)|^2|\ft{u}(\xi-\eta)|^2d\eta+\int_0^{\xi}(\xi-\eta)|\ft{u}(\eta)|^2|\ft{u}(\xi-\eta)|^2d\eta\Big{)}.
\end{align*}

\noi
By change of variables $\xi-\eta\mapsto \eta$ in the second integral, we have
$$|\ft{u^2}(\xi)|^2\leq \frac{1}{2\pi^2}\int_0^{\xi}\eta|\ft{u}(\eta)|^2|\ft{u}(\xi-\eta)|^2d\eta.$$
By Fubini's theorem and  change of variables $\zeta=\xi-\eta$ it results that
\begin{equation*}
E\leq \frac{1}{4\pi^3}\int_{\R}\int_0^{\xi}\eta|\ft{u}(\eta)|^2|\ft{u}(\xi-\eta)|^2d\eta d\xi=\frac{1}{4\pi^3}\int_0^{+\infty}\eta|\ft{u}(\eta)|^2d\eta\int_0^{+\infty}|\ft{u}(\zeta)|^2d\zeta=\frac{1}{\pi} MQ.
\end{equation*}

Moreover, equality holds if and only if we have equality in  Cauchy-Schwarz inequality, i.e.
\begin{equation*}
\ft{u}(\xi)\ft{u}(\eta)=\ft{u}(\xi+\eta)\ft{u}(0),
\end{equation*}

\noi
for all $\xi,\eta\geq 0$. This is true if and only if $\ft{u}(\xi)=e^{-ip\xi}\ft{u}(0)$, for all $\xi\geq 0$. Since $u\in H^{1/2}_+$, this yields $\text{Im}(p)<0$ and $u(x)=\frac{C}{x-p}$, for some constant $C$.
\end{proof}

The second argument we use in proving stability of traveling waves is a profile decomposition theorem.  It states that bounded sequences in $H^{1/2}_+$ can be written as superposition of translations of fixed profiles and of a remainder term. The remainder is small in all the $L^p$-norms, $2<p<\infty$. Moreover, the superposition is almost orthogonal in the $H_+^{1/2}$-norm.

\begin{proposition}[The profile decomposition theorem for bounded sequences in $H^{1/2}_+$]\label{prop: profile dec}
Let $\{v^n\}_{n\in\N}$ be a bounded sequence in $H^{1/2}_+$. Then, there exist a subsequence of $\{v^n\}_{n\in\N}$,
still denoted by $\{v^n\}_{n\in\N}$, a sequence of fixed profiles in $H^{1/2}_+$, $\{V^{(j)}\}_{j\in\N}$, and a family of real sequences $\{x^{(j)}\}_{j\in\N}$ such that for all $\ell\in\N^{\ast}$ we have
\begin{equation*}
v^n=\sum_{j=1}^{\ell}V^{(j)}(x-x_n^{(j)})+r_{n}^{(\ell)},
\end{equation*}

\noi
where
\begin{equation*}
\lim_{\ell\to\infty}\limsup_{n\to\infty}\|r_{n}^{(\ell)}\|_{L^p(\R)}=0
\end{equation*}

\noi
for all $p\in (2,\infty)$, and

\noi
\begin{align*}
\|v^n\|^2_{L^2} & =\sum_{j=1}^{\ell}\|V^{(j)}\|^2_{L^2}+\|r_n^{(\ell)}\|_{L^2}^2+o(1),\ \text{  as } n\to \infty,\\
\|v^n\|^2_{\dot{H}^{1/2}_+} & =\sum_{j=1}^{\ell}\|V^{(j)}\|^2_{\dot{H}^{1/2}_+}+\|r_n^{(\ell)}\|_{\dot{H}^{1/2}_+}^2+o(1),\ \text{  as } n\to \infty, \\
\lim_{n\to \infty}\|v^n\|^4_{L^4} & =\sum_{j=1}^{\infty}\|V^{(j)}\|^4_{L^4}.
\end{align*}
\end{proposition}

The proof of this proposition follows exactly the same lines as that of the profile decomposition theorem for bounded sequences in $H^{1}(\R)$, \cite[Proposition 2.1]{Keraani}. However, note that in our case, the profiles $V^{(j)}$ belong to the space $H^{1/2}_+$, (not only to the space $H^{1/2}(\R)$), as they are weak limits of translations of the sequence $\{v^n\}_{n\in\N}$.

\begin{proof}[Proof of Corollary \ref{main cor}]
According to Proposition \ref{prop:GN}, $C(a,r)$ is the set of minimizers of the problem
\begin{equation*}
\inf\{M(u)\big{|} u\in H^{1/2}_+, Q(u)=q(a,r), E(u)=e(a,r)\},
\end{equation*}

\noi
where
\begin{equation*}
q(a,r)=\frac{a^2\pi}{r},\ \ \ \ \ e(a,r)=\frac{a^4\pi}{2 r^3}.
\end{equation*}

\noi
We denote the infimum by $m(a,r)$.

Since
\begin{equation*}
\inf_{\phi\in C(a,r)}\|u_0^n-\phi\|_{H^{1/2}_+}\to 0,
\end{equation*}

\noi
by the Sobolev embedding theorem, we deduce
\begin{equation*}
Q(u_0^n)\to q(a,r),\ \ \ \ \ \ E(u_0^n)\to e(a,r),\ \ \ \ \ \ M(u_0^n)\to m(a,r).
\end{equation*}

\noi
Let $\{t_n\}_{n\in\N}$ be an arbitrary sequence of real numbers.
The conservation laws yield
\begin{equation*}
Q(u^n(t_n))\to q(a,r),\ \ \ E(u^n(t_n))\to e(a,r),\ \ \ M(u^n(t_n))\to m(a,r).
\end{equation*}

\noi
We can choose two sequences of positive numbers $\{a_n\}$ and $\{\ld_n\}$ such that $v^n(x):=a_nu^n(t_n,\ld_n x)$ satisfies $\|v^n\|_{L^{2}(\R)}=1$, $\|v^n\|_{L^{4}(\R)}=1$. Notice that
\[a_n\to a_{\infty},\,\,\,\,\, \ld_n\to \ld_{\infty},\]
where $a_{\infty}>0$, $\ld_{\infty}>0$, and
\[\frac{\ld_{\infty}}{a_{\infty}^4}=e(a,r),\,\,\,\,\,\, \frac{\ld_{\infty}}{a_{\infty}^2}=q(a,r).\]
Then
\begin{equation*}
\|v^n\|^{1/2}_{\dot{H}^{1/2}_+}=\frac{\|v^n\|^{1/2}_{L^{2}}\|v^n\|^{1/2}_{\dot{H}^{1/2}_+}}{\|v^n\|_{L^{4}}}=\frac{\|u^n(t_n)\|^{1/2}_{L^{2}}\|u^n(t_n)\|^{1/2}_{\dot{H}^{1/2}_+}}{\|u^n(t_n)\|_{L^{4}}},
\end{equation*}

\noi
for all $n\in\N$.
In particular, as a consequence of the Gagliardo-Nirenberg inequality,
$$\lim_{n\to \infty}\|v^n\|_{\dot{H}^{1/2}_+}=\sqrt{\pi}
.$$
Thus, the sequence $\{v^n\}_{n\in\N}$ is bounded in $H^{1/2}_+$. Applying the profile decomposition theorem
(Proposition \ref{prop: profile dec}), we obtain that there exist real sequences $\{x^{(j)}\}_{j\in\N}$ depending on the sequence $\{t_n\}_{n\in\N}$ in the definition of $\{v^n\}_{n\in\N}$, such that for all $\ell\in\N^{\ast}$ we have:
\begin{equation*}
v^n=\sum_{j=1}^{\ell}V^{(j)}(x- x_n^{(j)})+r_{n}^{(\ell)},
\end{equation*}

\noi
where
\begin{equation*}
\lim_{\ell\to\infty}\limsup_{n\to\infty}\|r_{n}^{(\ell)}\|_{L^p(\R)}=0
\end{equation*}

\noi
for all $p\in (2,\infty)$, and

\noi
\begin{align*}
\|v^n\|^2_{L^2} & =\sum_{j=1}^{\ell}\|V^{(j)}\|^2_{L^2}+\|r_n^{(\ell)}\|_{L^2}^2+o(1),\ \text{  as } n\to \infty,\\
\|v^n\|^2_{\dot{H}^{1/2}_+} & =\sum_{j=1}^{\ell}\|V^{(j)}\|^2_{\dot{H}^{1/2}_+}+\|r_n^{(\ell)}\|_{\dot{H}^{1/2}_+}^2+o(1),\ \text{  as } n\to \infty, \\
\lim_{n\to \infty}\|v^n\|^4_{L^4} & =\sum_{j=1}^{\infty}\|V^{(j)}\|^4_{L^4}.
\end{align*}

\noi
Consequently,
\begin{equation}\label{eqn:profile 1}
1\geq\sum_{j=1}^{\infty}\|V^{(j)}\|^2_{L^2},\quad \quad
\pi\geq\sum_{j=1}^{\infty}\|V^{(j)}\|^2_{\dot{H}^{1/2}_+},\quad \quad
1=\sum_{j=1}^{\infty}\|V^{(j)}\|^4_{L^4}.
\end{equation}

\noi
Therefore, by the Gagliardo-Nirenberg inequality \eqref{ineqn:Gagliardo-Nirenberg}, we have
\begin{align*}
\pi
\geq (\sum_{j=1}^{\infty}\|V^{(j)}\|^2_{L^2})(\sum_{j=1}^{\infty}\|V^{(j)}\|^2_{\dot{H}^{1/2}_+})
\geq \sum_{j=1}^{\infty}\|V^{(j)}\|^2_{L^2}\|V^{(j)}\|^2_{\dot{H}^{1/2}_+}
\geq \pi\sum_{j=1}^{\infty}\|V^{(j)}\|^4_{L^4}=\pi.
\end{align*}

\noi
Thus, there exist only one profile $V:=V^{(1)}$ and a sequence $x=x^{(1)}$ such that
\begin{align}
v^n=&V(x- x_n)+r_{n},\notag\\
\|v^n\|^2_{L^2}=&\|V\|^2_{L^2}+\|r_n\|_{L^2}^2+o(1),\ \text{  as } n\to \infty,\label{profile 2}\\
\|v^n\|^2_{\dot{H}^{1/2}_+}=&\|V\|^2_{\dot{H}^{1/2}_+}+\|r_n\|_{\dot{H}^{1/2}_+}^2+o(1),\ \text{  as } n\to \infty.\label{profile 3}
\end{align}

\noi
According to \eqref{eqn:profile 1}, $V$ satisfies $1\geq\|V\|_{L^{2}}^2$, $\pi\geq\|V\|^2_{\dot{H}^{1/2}_+}$, and $\|V\|^4_{L^4}=1$.
In conclusion,
\begin{equation*}
\pi=\pi\|V\|_{L^{4}}^4\leq  \|V\|_{L^2}^{2}\|V\|^2_{\dot{H}^{1/2}_+}\leq \pi.
\end{equation*}

\noi
Hence, $V$ is a minimizer in the Gagliardo-Nirenberg inequality.
Moreover, \[\|V\|_{L^{2}}^2=1=\|v^n\|_{L^{2}},\ \ \ \ \ \ \ \ \ \ \|V\|^2_{\dot{H}^{1/2}_+}=\pi=\lim_{n\to \infty}\|v^n\|_{\dot{H}^{1/2}_+}^2,\]

\noi
By \eqref{profile 2} and \eqref{profile 3}, we have $r_n\to 0$ in $H^{1/2}_+$ as $n\to \infty$. Consequently, $v^n(\cdot+x_n)\to V$ in $H^{1/2}_+$, or equivalently,
\begin{equation*}
\lim_{n\to \infty}\|a_nu^n(t_n,\ld_nx)-V(x-x_n)\|_{H^{1/2}_+}=0.
\end{equation*}

\noi
We then have
\begin{equation*}
\lim_{n\to \infty}\|u^n(t_n,x)-\frac{1}{a_{\infty}}V(\frac{x- x_n\ld_{\infty}}{\ld_{\infty}})\|_{H^{1/2}_+}=0.
\end{equation*}

\noi
Notice that, since $V$ is a minimizer in the Gagliardo-Nirenberg inequality,
we have $\tilde\phi(x):=\frac{1}{a_{\infty}}V(\frac{x}{\ld_{\infty}})=\frac{\alpha}{x-p}\in C(a,r)$. Then, since $x_n\lambda_{\infty}\in\R$,
we have $\phi(x)=\tilde\phi(x-x_n\ld_{\infty})=\frac{\alpha}{x-\tilde p}\in C(a,r)$. Thus,
\begin{equation}\label{sup approx}
\inf_{\phi\in C(a,r)}\|u^n(t_n,x)-\phi(x)\|_{H^{1/2}_+}\to 0, \text{ as } n\to\infty.
\end{equation}

\noi
The conclusion follows by approximating the supremum in the statement by the sequence in \eqref{sup approx} with an appropriate $\{t_n\}_{n\in\N}$.
\end{proof}

\noi
{\bf Acknowledgments:}
The author is grateful to her Ph.D. advisor Prof. Patrick G\'erard
for introducing her to this subject and for constantly supporting her during the preparation of this paper.
She would also like to thank the referee for his helpful comments.

\end{document}